\documentclass[a4paper,12pt]{amsart}
\usepackage{amssymb}
\usepackage{cite}
\usepackage{ifthen}
\usepackage[dvips]{graphicx}
\nonstopmode \numberwithin{equation}{section}
\usepackage{amssymb}
\usepackage{ifthen}
\usepackage{graphicx}
\usepackage{amsmath}
\usepackage[T1]{fontenc} 
\usepackage[utf8]{inputenc}
\usepackage[usenames,dvipsnames]{color}
\usepackage{color}
\usepackage[english]{babel}
\usepackage{fancyhdr}
\usepackage{fancybox}
\usepackage{tikz}

\theoremstyle{plain}
\newtheorem{prop}{Proposition}

\newtheorem{ques}{Question}

\newtheorem{conj}{Conjecture}

\theoremstyle{definition}
\newtheorem{defi}{Definition}[section]

\newtheorem{cor}{Corollary}[section]
\newtheorem{thm}{Theorem}[section]

\newtheorem{lem}{Lemma}[section]
\newtheorem{prob}{Problem}
\newtheorem{rem}{Remark}[section]

\theoremstyle{plain}
\newtheorem*{thmA}{Theorem A}
\newtheorem*{thmB}{Theorem B}
\newtheorem*{thmC}{Theorem C}
\newtheorem*{thmD}{Theorem D}

\newtheorem*{lemA}{Lemma A}
\newtheorem*{lemB}{Lemma B}

\newcounter{minutes}
\divide\time by 60
\newcounter{hours}
\multiply\time by 60
\addtocounter{minutes}{-\time}

\newcounter {own}
\def\theown {\thesection       .\arabic{own}}

\newenvironment{pf}[1][]{%
	\vskip 3mm
	\noindent
	\ifthenelse{\equal{#1}{}}%
	{{\slshape Proof. }}%
	{{\slshape #1.} }%
}%
{\qed\bigskip}

\newcounter{alphabet}

\def\be{\begin{equation}}
	\def\ee{\end{equation}}

\newcommand{\bee}{\begin{enumerate}}
	\newcommand{\eee}{\end{enumerate}}

\newcommand{\blem}{\begin{lem}}
	\newcommand{\elem}{\end{lem}}
\newcommand{\bthm}{\begin{thm}}
	\newcommand{\ethm}{\end{thm}}
\newcommand{\bcor}{\begin{cor}}
	\newcommand{\ecor}{\end{cor}}
\newcommand{\beg}{\begin{examp}}
	\newcommand{\eeg}{\end{examp}}
\newcommand{\begs}{\begin{examples}}
	\newcommand{\eegs}{\end{examples}}

\newcommand{\bdefn}{\begin{defn}}
	\newcommand{\edefn}{\end{defn}}

\newcommand{\bprob}{\begin{prob}}
	\newcommand{\eprob}{\end{prob}}
\newcommand{\bei}{\begin{itemize}}
	\newcommand{\eei}{\end{itemize}}

\newcommand{\bcon}{\begin{conj}}
	\newcommand{\econ}{\end{conj}}
\newcommand{\bcons}{\begin{conjs}}
	\newcommand{\econs}{\end{conjs}}
\newcommand{\bprop}{\begin{prop}}
	\newcommand{\eprop}{\end{prop}}
\newcommand{\br}{\begin{rem}}
	\newcommand{\er}{\end{rem}}
\newcommand{\brs}{\begin{rems}}
	\newcommand{\ers}{\end{rems}}
\newcommand{\bo}{\begin{obser}}
	\newcommand{\eo}{\end{obser}}
\newcommand{\bos}{\begin{obsers}}
	\newcommand{\eos}{\end{obsers}}
\newcommand{\bpf}{\begin{pf}}
	\newcommand{\epf}{\end{pf}}
\newcommand{\ba}{\begin{array}}
	\newcommand{\ea}{\end{array}}
\newcommand{\beq}{\begin{eqnarray}}
	\newcommand{\beqq}{\begin{eqnarray*}}
		\newcommand{\eeq}{\end{eqnarray}}
	\newcommand{\eeqq}{\end{eqnarray*}}

\begin{document}

\title{The Bohr radius for operator valued functions on simply connected domain}

\author{Sabir Ahammed}
\address{Sabir Ahammed, Department of Mathematics, Jadavpur University, Kolkata-700032, West Bengal,India.}
\email{sabira.math.rs@jadavpuruniversity.in}

\author{Molla Basir Ahamed}
\address{Molla Basir Ahamed, Department of Mathematics, Jadavpur University, Kolkata-700032, West Bengal,India.}
\email{mbahamed.math@jadavpuruniversity.in}

\subjclass[{AMS} Subject Classification:]{30A10, 30B10, 30C50, 30C55, 30F45, 30H30, 46E40, 47A56, 47A63}
\def\thefootnote{}
\footnotetext{ {\tiny File:~\jobname.tex,
printed: \number\year-\number\month-\number\day,
          \thehours.\ifnum\theminutes<10{0}\fi\theminutes }
} \makeatletter\def\thefootnote{\@arabic\c@footnote}\makeatother
\begin{abstract} 
In this paper, we first establish an improved Bohr inequality for the class of operator-valued holomorphic functions $f$ on a simply connected domain $\Omega$ in $\mathbb{C}$. Next, we establish a generalization of refined version of the Bohr inequality and the Bohr-Rogosinski inequality with the help of the sequence $\varphi=\{\varphi_n(r) \}^{\infty}_{n=0}$ of non-negative continuous functions in $[0,1)$ such that the series $\sum_{n=0}^{\infty}\varphi_n(r)$ converges locally uniformly on the interval $[0,1)$. All the results are proved to be sharp.  Moreover, We establish the Bohr inequality and the Bohr-Rogosinski inequality for the class of operator-valued $\nu$-Bloch functions defined in two different simply connected domains, $\Omega$ and $\Omega_{\gamma}$, in $\mathbb{C}$. 
\end{abstract}
\maketitle
\pagestyle{myheadings}
\markboth{S. Ahammed and M. B. Ahamed}{The Bohr radius for operator valued functions on simply connected domain}
\tableofcontents
\section{\bf Introduction}
Let $\mathbb{D}(a;r):=\{z\in \mathbb{C}: |z-a|<r\}$, and let $\mathbb{D}:=\mathbb{D}(0;1)$ be the open unit disk in the complex plane $\mathbb{C}$. Also, let $\mathbb{N}_0$ be the set of non-negative integers. For a given simply connected domain $\Omega$ containing the unit disk $\mathbb{D},$ we denote $\mathcal{H}\left(\Omega\right)$ a class of analytic functions on $\Omega$, and $\mathcal{B}\left(\Omega\right)$ be the class of functions $f\in\mathcal{H}\left(\Omega\right) $ such that $f\left(\Omega\right)\subseteq \mathbb{D}.$ The Bohr radius for the family $\mathcal{B}\left(\Omega\right)$ is defined to be the positive real number $B_{\Omega}\in (0,1)$ given by (see \cite{Fournier-2010}) 
\begin{align*}
	B_{\Omega}=\sup \{r\in (0,1): M_f(r)\leq 1\; \mbox{for all}\;f(z)=\sum_{n=0}^{\infty}a_nz^n\in\mathcal{B}\left(\Omega\right), z\in \mathbb{D} \},
\end{align*}
where $M_f(r):=\sum_{n=0}^{\infty}|a_n|r^n$ is the associated majorant series for $f\in \mathcal{B}\left(\Omega\right)$. If $\Omega=\mathbb{D},$ then it is well-known that $B_{\mathbb{D}}=1/3,$ and it is described in \cite{Bohr-1914} precisely as follows.
\begin{thmA}\emph{\cite{Bohr-1914}}
If $f(z)=\sum_{n=0}^{\infty}a_nz^n\in \mathcal{B}\left(\mathbb{D}\right)$, then 
\begin{align}\label{Eq-1.1}
\sum_{n=0}^{\infty}|a_n|r^n\leq 1\,\, \mbox{for}\;\;|z|=r\leq \dfrac{1}{3}.
\end{align}
 The number $1/3$ is best possible.
\end{thmA}
The inequality $\eqref{Eq-1.1}$ is known as the classical Bohr inequality, while $1/3$ is known as the Bohr radius. Bohr's theorem gained popularity when Dixon \cite{Dixon & BLMS & 1995} used it to disprove a conjecture that stated if the non-unital von Neumann inequality holds for a Banach algebra, then it must be an operator algebra. In 2000, Djkaov and Ramanujan \cite{Djakov-Ramanujan-JA-2000} conducted an extensive study on the best possible constant $r_p$ for $1\leq p<\infty$, where $r_p$ is defined as follows: 
\begin{align*}
	\left(\sum_{n=0}^{\infty}|a_n|^p(r_p)^{np}\right)^{1/p}\leq ||f||_{\infty},
\end{align*}
where $f(z)=\sum_{n=0}^{\infty}a_nz^n.$ For $p=1,$ $r_p$ coincides with the classical Bohr radius $1/3.$ Using Haussdorf-Young's inequality, it is easy to see that $r_p=1$ for $p\in [2,\infty)$. However, computing the precise value of $r_p$ for $1<p<2$ is generally difficult. Therefore, there is a need to estimate the value of $r_p$. In \cite{Djakov-Ramanujan-JA-2000}, the following best-known estimate 
\begin{align*}
	\left(1+\left(\dfrac{2}{p}\right)^{\frac{1}{2-p}}\right)^{\frac{p-2}p}\leq r_p\leq \inf_{0\leq a<1}\left(\dfrac{(1-a^p)^{1/p}}{((1-a^2)^p+a^p(1-a^p))^{1/p}}\right)
\end{align*}
 has been obtained.\vspace{1.2mm} 
 
 Paulsen \emph{et al.} \cite{Paulsen-Popescu-Singh-PLMS-2002} have considered another modification of $\eqref{Eq-1.1}$ for $f(z)=\sum_{n=0}^{\infty}a_nz^n\in \mathcal{B}\left(\mathbb{D}\right)$ and have shown that 
\begin{align}\label{Eq-1.2}
	|a_0|^2+\sum_{n=1}^{\infty}|a_n|r^n\leq 1\;\;\mbox{for}\;\;|z|= r\leq \dfrac{1}{2},
\end{align}
 and the constant $1/2$ is sharp. Kayumov and Ponnusamy  \cite{Kayumov-Ponnusamy-CRACAD-2018} have examined a different variation of $\eqref{Eq-1.1}$ for the function $f(z)=\sum_{n=0}^{\infty}a_nz^n\in \mathcal{B}\left(\mathbb{D}\right)$. They have shown that
 \begin{align}\label{Eq-1.3}
 	|a_0|+\sum_{n=1}^{\infty}\left(|a_n|r^n+\dfrac{1}{2}|a_n|^2\right)\left(\dfrac{1}{3}\right)^n\leq 1
 \end{align}
 and the constants $1/2$ and $1/3$ cannot be improved. \vspace{1.2mm}
 
In  \cite{Ponnusamy-Vijayakumar-Wirths-RM-2020}, Ponnusamy \emph{et al.} explored a modified version of $\eqref{Eq-1.1}$ known as a refinement of the classical Bohr inequality. They present their findings that
 \begin{align}\label{Eq-1.4}
 	\sum_{n=0}^{\infty}|a_n|r^n+\left(\frac{1}{1+|a_0|}+\frac{r}{1-r}\right)\sum_{n=1}^{\infty}|a_n|^2r^{2n}\leq 1 \quad \mbox{for} \quad r \leq \frac{1}{3}
 \end{align}
 and the constant $1/3$ is sharp.\vspace{1.2mm}
 
 In addition to the Bohr inequality, there is another concept called the Bohr-Rogosinski inequality. The Rogosinski radius, which is compared to the Bohr radius, was first introduced in \cite{Rogosinski-1923} for functions $f\in\mathcal{B}(\mathbb{D})$. However, the Rogosinski radius has not been studied as extensively as the Bohr radius. The following paragraph presents the definition of the Rogosinski radius.
 \begin{thmB}\emph{(Rogosinski Theorem)}
 	If $ g(z)=\sum_{n=0}^{\infty}b_nz^n\in\mathcal{B}(\mathbb{D}),$  then for every $ N\in \mathbb{N} $, 
 	\begin{align*}
 		\bigg|\sum_{n=0}^{N-1}b_nz^n\bigg|\leq 1\; \mbox{for}\; |z|=r\leq \frac{1}{2}.
 	\end{align*}
 	The constant $ 1/2 $ is best possible and is called Rogosinski radius.
 \end{thmB}
 In \cite{Kayumov-Khammatova-Ponnusamy-JMAA-2021},
 the Bohr-Rogosinski sum $ R^f_N(z) $ for the functions $ f(z)=\sum_{n=0}^{\infty}a_nz^n\in\mathcal{B}(\mathbb{D}), $ is defined by  
 \begin{align}\label{e-0.2}
 	R^f_N(z):=|f(z)|+\sum_{n=N}^{\infty}|a_n|r^n, \; |z|=r.
 \end{align}
 An interesting observation is that for $N=1$, the quantity in \eqref{e-0.2} is linked to the classical Bohr sum, where $|f(0)|$ is replaced by $|f(z)|$. This quantity, denoted as $R^f_N(z)$, satisfies the Bohr-Rogosinski inequality $R^f_N(z)\leq 1$. The Bohr radius, denoted as $B$, and the Bohr-Rogosinski radius, denoted as $R$, satisfy the inequality $B = 1/3 < 1/2 = R$. For recent developments on the Bohr-Rogosinski phenomenon, readers are referred to the articles \cite{Das-JMAA-2022,Aizeberg-Elin-Shoik-SM-2005,Aizen-AMP-2012,Allu-Arora-JMAA-2022},  and references therein. For other aspect of Bohr phenomenon, we refer tot he articles \cite{Aha-Allu-RMJ-2022,Ahamed-Allu-Halder-AASFM-2022,Aizn-PAMS-2000,Aytuna-Djakov-BLMS-2013,Ali & Abdul & Ng & CVEE & 2016,Alkhaleefah-Kayumov-Ponnusamy-PAMS-2019,Boas-Khavinson-PAMS-1997,Defant-AM-2012,Galicer-TAMS-2020,Kumar-PAMS-2022,Kumar-Manna-JMAA-2023,Ponnusamy-Shmidt-Starkov-JMAA-2024} and references therein. 
\section{\bf The Bohr inequality on a simply connected domain}
It is worth noting that the concept of Bohr's radius, which was originally defined for analytic functions from the unit disk $\mathbb{D}$ into $\mathbb{D}$, has been extended by authors to include mappings from $\mathbb{D}$ into other domains $\Omega$ in $\mathbb{D}$ (see e.g., \cite{Muhanna-Ali-JMAA-2011,Aizen-SM-2007}). For our study, we consider $ \Omega $ a simply connected domain in $\mathbb{C}$ which contains the unit disk $ \mathbb{D} $ and bounded holomorphic functions from $\Omega$ into a complex Banach space $X$.\vspace{1.2mm}	

 Let $H^{\infty}\left(\Omega,X\right)$ be the space of bounded holomorphic functions from $\Omega$ into a complex Banach space $X$ and 
$||f||_{H^{\infty}\left(\Omega,X\right)}:=\sup_{z\in \Omega}||f(z)||.$ For $p\in [1,\infty),$ $H^p(\mathbb{D}, X)$ denotes the space of analytic functions from $\mathbb{D}$ into $X$ with the norm
\begin{align*}
	||f||_{H^p(\mathbb{D}, X)}=\sup_{0<r<1}\left(\dfrac{1}{2\pi}\int_{0}^{2\pi}||f(re^{it})||^pdt\right)^{1/p}<\infty.
\end{align*}
 In this paper, $\mathcal{B}(\mathcal{H})$ denotes the spaces of bounded linear operators on complex Hilbert space $\mathcal{H}.$ For any $A\in \mathcal{B}(\mathcal{H}),$ $||A||$ denotes the operator norm of $A.$ Let $A\in \mathcal{B}(\mathcal{H}).$ Then the adjoint operator $A^*:\mathcal{H}\rightarrow \mathcal{H} $ of $A$ defined by $(Ax,y)=(x,A^*y)$ for all $x,y\in \mathcal{H}.$  The operator $A$ is said to be normal if $A^*A=AA^*,$ self-adjoint  if $A^*=A,$ and positive if $(Ax,x)\geq 0$ for all $x\in \mathcal{H}.$ The absolute value of $A$ is defined by $|A|:= (A^*A)^{1/2},$ where $T^{1/2}$ denotes the unique positive square root of a positive operator $T.$ Let $I$ be the identity operator on $\mathcal{H}.$\vspace{1.2mm} 
 
We denote (see \cite{Allu-Halder-PEMS-2023}) the constant $\lambda_H$ by 
\begin{equation}\label{e-2.3}
\lambda_H:=\lambda_H(\Omega)=\sup_{f\in {H}^{\infty}\left(\Omega, \mathcal{B}(\mathcal{H})\right),\;||f(z)||\leq 1 }\bigg\{\frac{||A_n||}{||I-|A_0|^2||} :  A_0\neq f(z)=\sum_{n=0}^{\infty}A_nz^n,\;\; \; z\in\mathbb{D}\bigg\}.
\end{equation}
Let $f$ be holomorphic in $\mathbb{D}$, and for $0<r<1$, let $\mathbb{D}_r=\{z\in\mathbb{C}: |z|<r\}$, and $S_r:=S_r(f)$ denotes the planar integral
\begin{align*}
S_r:=\int_{\mathbb{D}_r}|f^{\prime}(z)|^2dA(z).
\end{align*} 
If $f(z)=\sum_{n=0}^{\infty}a_nz^n\in\mathcal{B}(\mathbb{D}), $ then 
\begin{align*}
S_r:=\pi\sum_{n=1}^{\infty}n|a_n|^2r^{2n}.
\end{align*}
If $f$ is a univalent function, then $S_r$ represents the area of $f(\mathbb{D}_r)$. However, in the case of a multivalent function, $S_r$ is greater than the area of the image of the subdisk $\mathbb{D}_r$. This fact could be shown by noting that 
\begin{align*}
S_r=\int_{\mathbb{D}_r}|f^{\prime}(w)|^2dA(w)=\int_{f(\mathbb{D}_r)}\nu_f(w)dA(w)\geq\int_{f(\mathbb{D}_r)}dA(w)=Area(f(\mathbb{D}_r)),
\end{align*}
where $\nu_f(w)=\sum_{f(z)=w}1$, denotes the counting function of $f$.\vspace{1mm}
	
In the study of the improved Bohr inequality, the quantity $S_r$ plays a significant role. There have been many results on the improved Bohr inequality for the class $\mathcal{B}\left(\mathbb{D}\right)$ (see e.g.,  \cite{Ismagilov-Kayumov-Ponnusamy-2020-JMAA, Kayumov-Ponnusamy-CRACAD-2018}), which includes harmonic mappings on the unit disk (\cite{Evdoridis-Ponnusamy-Rasila-IM-2018}). Kayumov and Ponnusamy \cite{Kayumov-Ponnusamy-CRACAD-2018} have considered another modification of $\eqref{Eq-1.1}$ and $ \eqref{Eq-1.2}$ for the class $ \mathcal{B}\left(\mathbb{D}\right)$ and have shown that
\begin{align*}
\sum_{n=0}^{\infty}|a_n|r^n+\frac{16}{9}\left(\frac{S_r}{\pi}\right)\leq 1\;\;\mbox{for}\;\;|z|= r\leq \dfrac{1}{3},
\end{align*}
and 
\begin{align*}
	|a_0|^2+\sum_{n=1}^{\infty}|a_n|r^n+\frac{9}{8}\left(\frac{S_r}{\pi}\right)\leq 1\;\;\mbox{for}\;\;|z|= r\leq \dfrac{1}{2}.
\end{align*}
It is not simply about using any combination of $S_r/\pi$ and its powers for the study. The key is to always remember that the inequality must be satisfied with the given radius, and not only that, but those radii must remain unchanged. Additionally, it is crucial to demonstrate that these radii are best possible and that the associated coefficients cannot be further improved. Therefore, the study presents a challenge as we must find the appropriate combination on the left-hand side of the inequality. In this regard, Ismagilov \emph{et al.} \cite{Ismagilov-Kayumov-Ponnusamy-2020-JMAA} remarked that for any function $F : [0, \infty)\to [0, \infty)$ such that $F(t)>0$ for $t>0$, there exists an analytic function $f : \mathbb{D}\to\mathbb{D}$ for which the inequality
\begin{align*}
	\sum_{n=0}^{\infty}|a_n|r^n+\frac{16}{9}\left(\frac{S_r}{\pi}\right)+\lambda\left(\frac{S_r}{\pi}\right)^2+F(S_r)\leq 1\;\; \mbox{for}\;\; r\leq\frac{1}{3}
\end{align*}
is false, where $\lambda$ is given in \cite[Theorem 1]{Ismagilov-Kayumov-Ponnusamy-2020-JMAA}.\vspace{1.2mm}

It is important to note that combining a non-negative quantity with the Bohr inequality does not always result in the desired inequality for the class $\mathcal{B}\left(\mathbb{D}\right)$.
However, we would like to point out that despite the extensive exploration of Bohr-type inequalities, the study of these inequalities for the class of functions $ f\in {H}^{\infty}\left(\Omega, \mathcal{B}(\mathcal{H})\right)$ given by $ f(z)=\sum_{n=0}^{\infty}A_nz^n $ in $\mathbb{D} $ has not received as much attention from researchers. This lack of attention serves as the primary motivation for this section. Our main goal is to fill this specific gap in the existing literature and contribute to understanding Bohr-type inequalities.
\vspace{1.2mm}

To serve our purpose, we consider a polynomial $P(w)$ of degree $m$ as follows:
\begin{equation}\label{e-2.3a}
P(w)=k_1w+k_2w^2+\cdots+k_mw^m,\;\;\mbox{where}\;\; k_j:=\left(\frac{1+\lambda_H}{1+2\lambda_H}\right)^{2j}
\end{equation} 
for $j=1, 2, \cdots, m,$ and show that the following sharp result can be obtained.
\begin{thm}\label{th-3.3}	
Let $ \Omega $ be a simply connected domain containing the unit disk $ \mathbb{D} $ and $ f\in {H}^{\infty}\left(\Omega, \mathcal{B}(\mathcal{H})\right)$ be given by $ f(z)=\sum_{n=0}^{\infty}A_nz^n $ in $\mathbb{D} $. Then
\begin{equation*}
\mathcal{A}_{\lambda_H}^f(r):=\sum_{n=0}^{\infty}||A_n||r^n+P\left(S_r\right)\leq 1\;\; \mbox{for}\;\; r\leq\frac{1}{1+2\lambda_H},
\end{equation*}
where $ S_r $ denotes the area of the image of the disk $ \mathbb{D}(0,r) $ under the mapping $ f $ and $ P(w) $ is a polynomial given by \eqref{e-2.3a}. The equality $\mathcal{A}_{\lambda_H}^f(r)=1 $ holds for the function $ f\in {H}^{\infty}\left(\Omega, \mathcal{B}(\mathcal{H})\right) $ if, and only if, $f$ is given by  $f(z)= cI $ with $ |c|=1. $
\end{thm}
\begin{rem}
It is important to note that the study of the improved Bohr inequality for the class $\mathcal{B}$ has certain limitations, as pointed out by Ismagilov \emph{et al.} in \cite{Ismagilov-Kayumov-Ponnusamy-2020-JMAA}. However, in Theorem \ref{th-3.3}, these limitations have been overcome for the class ${H}^{\infty}\left(\Omega, \mathcal{B}(\mathcal{H})\right)$.
\end{rem}
One concern addressed in this article is the exploration of several related questions about the Bohr inequality. Specifically, the aim is to investigate an inequality that holds for the class $\mathcal{B}$ and examine if it can also be established for other classes.\vspace{1.2mm} 

Therefore, a natural question arises.
\begin{ques}\label{Eq-2.1}
Is it possible to establish a sharp analog of the inequalities \eqref{Eq-1.3} and \eqref{Eq-1.4} for the class ${H}^{\infty}\left(\Omega, \mathcal{B}(\mathcal{H})\right)$?
\end{ques} 
In our response to Question \ref{Eq-2.1}, we establish an analogue of the inequalities \eqref{Eq-1.3} and \eqref{Eq-1.4} for the class of functions $f \in {H}^{\infty}\left(\Omega, \mathcal{B}(\mathcal{H})\right)$. Consequently, we obtain Theorem \ref{th-3.4} and Theorem \ref{th-3.5} below, respectively. Moreover, we show that the equality holds in each result for $f$ is given by  $f(z)= cI $ with $ |c|=1.$
\begin{thm}\label{th-3.4}
Let $ \Omega $ be a simply connected domain containing the unit disk $ \mathbb{D} $ and  $ f\in {H}^{\infty}\left(\Omega, \mathcal{B}(\mathcal{H})\right)$ be given by $ f(z)=\sum_{n=0}^{\infty}A_nz^n $ in $\mathbb{D} $. Then for $ 0\leq\beta\leq 1/{4\lambda_H}$,
\begin{equation*}
\mathcal{B}_{\lambda_H}^f(\beta,r):=||A_0||+\sum_{n=1}^{\infty}\bigg(||A_n||+\beta ||A_n||^2\bigg)r^n\leq 1\;\; \mbox{for}\;\; r\leq\frac{1}{1+2\lambda_H}.
\end{equation*}
The equality $ \mathcal{B}_{\lambda_H}^f(\beta,r)=1 $ holds for the function $ f\in {H}^{\infty}\left(\Omega, \mathcal{B}(\mathcal{H})\right) $ if, and only if, $f$ is given by  $f(z)= cI $ with $ |c|=1. $
\end{thm}
\begin{thm}\label{th-3.5}
Let $ \Omega $ be a simply connected domain containing the unit disk $ \mathbb{D} $ and $ f\in {H}^{\infty}\left(\Omega, \mathcal{B}(\mathcal{H})\right)$ be given by $ f(z)=\sum_{n=0}^{\infty}A_nz^n $ in $\mathbb{D} $.  Then 
\begin{equation*}
\mathcal{C}_{\lambda_H}^f(r):=\sum_{n=0}^{\infty}||A_n||r^n+\left(\frac{1+\lambda_H}{2\lambda_H(1+||A_0||)}+\frac{2(1+\lambda_H)r}{3(1-r)}\right)\sum_{n=1}^{\infty}||A_n||^2r^{2n}\leq 1
\end{equation*}
 for $r\leq{1}/{(1+2\lambda_H)}$. The equality $\mathcal{C}_{\lambda_H}^f(r)=1$ holds for the function $ f\in {H}^{\infty}\left(\Omega, \mathcal{B}(\mathcal{H})\right) $ if, and only if, $f$ is given by  $f(z)= cI $ with $ |c|=1. $
\end{thm}
\subsection{\bf The Bohr inequality involving a sequence $\varphi=\{\varphi_n(r) \}^{\infty}_{n=0}$ of non-negative continuous function in $[0,1)$}
For further discussions, we need to introduce some basic notations. Following a recent study on the Bohr inequality (see \cite{Kayumov-Khammatova-Ponnusamy-MJM-2022}), we will use the notation $\mathcal{F}$ to represent the set of all sequences $\varphi=\{\varphi_n(r) \}^{\infty}_{n=0}$ of non-negative continuous functions on the interval $[0,1)$ such that the series $\sum_{n=0}^{\infty}\varphi_n(r)$ converges locally uniformly on that interval.  The Bohr inequalities have been studied by numerous authors through the substitution of $ r^n $ with $ \varphi_n(r) \in \mathcal{F} $ (see e.g.,  \cite{Ponnusamy-Vijayakumar-Wirths-RM-2020,Kayumov-Khammatova-Ponnusamy-MJM-2022,Chen-Liu-Ponnusamy-RM-2023}). For the sake of convenience, we will use the notation
\begin{align*}
\Phi_N(r):=\sum_{n=N}^{\infty}\varphi_n(r)\;\;\mbox{and}\;\; 
\mathcal{A}\left(f_m,\varphi,r\right):=\sum_{n=m+1}^{\infty}||A_n||^{2n}\left(\dfrac{\varphi_{2n}(r)}{1+||A_m||}+\Phi_{2n+1}(r)\right),
\end{align*}
 where $f_m(z):=\sum_{n=m}^{\infty}A_nz^n.$ In particular, when $\varphi_n(r)=r^n,$ the formula for $\mathcal{A}\left(f_0,\varphi,r\right)$ takes the following simple form 
\begin{align*}
\mathcal{A}\left(f_0,r\right):=\left(\dfrac{1}{1+||A_0||}+\dfrac{r}{1-r}\right)\sum_{n=1}^{\infty}||A_n||^{2n}r^{2n},
\end{align*}
where $f_0(z)=f(z)-f(0),$ which measure helps to transform the traditional Bohr inequality into a more improved version.\vspace{1.2mm}
		
For $ 0\leq \gamma<1 $, we consider the simply connected domain  $ \Omega_{\gamma} $ defined by 
	\begin{align*}
		\Omega_{\gamma}:=\biggl\{z\in \mathbb{C}: \bigg|z+\dfrac{\gamma}{1-\gamma}\bigg|<\dfrac{1}{1-\gamma}\biggr\}.
	\end{align*}
It is clear that $ \Omega_{\gamma} $ contains the unit disk $ \mathbb{D} $, \textit{i.e.,} $\mathbb{D}\subseteq \Omega_{\gamma}$ for $0\leq \gamma<1$.\vspace{1.2mm}  
	
Fournier and Ruscheweyh (see \cite{Fournier-2010}) studied the classical Bohr inequality for the class $\mathcal{B}(\Omega_\gamma)$ and showed that it holds for the radius $\rho_\gamma=(1+\gamma)/(3+\gamma)$, resulting in particular $\rho_0=1/3$, which is the classical Bohr radius. In \cite{Evdoridis-Ponnusamy-Rasila-RM-2021}, the authors have studied improved Bohr inequalities for simply connected domain. Recently, in \cite{Allu-Halder-CMB-2023}, authors have studied the Bohr inequality for the class $ f\in {H}^{\infty}\left(\Omega_{\gamma}, \mathcal{B}(\mathcal{H})\right)$ and obtain the following result.
\begin{thmC} \emph{(\cite{Allu-Halder-CMB-2023})}
For $0\leq\gamma<1,$ let $ f\in {H}^{\infty}\left(\Omega_{\gamma}, \mathcal{B}(\mathcal{H})\right)$ be given by $ f(z)=\sum_{n=0}^{\infty}A_nz^n $ in $\mathbb{D} $ with $||f(z)||_{{H}^{\infty}\left(\Omega_{\gamma}, \mathcal{B}(\mathcal{H})\right)}\leq 1.$ Then
\begin{align*}
\sum_{n=0}^{\infty}||A_n||r^n\leq1 \;\;\mbox{for}\;\;r\leq\rho_\gamma:=\frac{1+\gamma}{3+\gamma}.
\end{align*}
The constant $\rho_\gamma$ is best possible.
\end{thmC}
The Bohr inequality has been extensively studied for Banach spaces on the domain $ \Omega_{\gamma} $. In particular, a sequence $ \{\varphi_n(r)\}_{n=0}^{\infty} \in \mathcal{F}$ has been used for this purpose (see e.g., \cite{Allu-Halder-PEMS-2023,Allu-Halder-CMB-2023}). But we observe some limitation of the results in \cite{Allu-Halder-PEMS-2023,Allu-Halder-CMB-2023}. In this paper, we will delve further into the exploration of the Bohr inequality. In order to further investigate on the Bohr inequality, we will now recall a definition.
\begin{defi}\emph{(\cite{Allu-Halder-PEMS-2023})}
Let  $ f\in {H}^{\infty}\left(\Omega, X\right)$ be given by $f(z)=\sum_{n=0}^{\infty}A_nz^n$ in $\mathbb{D}$ with $||f(z)||_{H^{\infty}\left( \Omega, \mathcal{B} (\mathcal{H})\right) }\leq 1.$ For $\varphi=\{\varphi_n(r)\}^{\infty}_{n=0}\in \mathcal{F}$ with $\varphi_{0}(r)\leq 1,$ $1\leq p,q<\infty,$ we denote
\begin{align*}
R_{p,q,\varphi}(f,\Omega,X)=\sup\biggl\{r\geq 0: ||A_0||^p\varphi_0(r)+\left(\sum_{n=1}^{\infty}||A_n||\varphi_n(r)\right)^q\leq \varphi_{0}(r)\biggr\}.
\end{align*}
\end{defi}
We define the Bohr radius associated with $\varphi$ by
\begin{align}\label{BS-eq-2.3}
R_{p,q,\varphi}\left(\Omega,X\right)=\inf\biggl\{	R_{p,q,\varphi}(f,\Omega,X):||f(z)||_{H^{\infty}\left( \Omega, \mathcal{B} (\mathcal{H})\right) }\leq 1\biggr\}.
\end{align}
For $p_1\leq p_2$ and $q_1\leq q_2,$ we have the following inclusion relation: 
\begin{align*}
		R_{p_1,q_1,\varphi}\left(\Omega,X\right)\leq 	R_{p_2,q_2,\varphi}\left(\Omega,X\right).
\end{align*}
It is established in \cite[Theorem 1.3]{Allu-Halder-PEMS-2023} that $R_{\Omega}(p)\leq R_{p,1,\varphi}\left(\Omega,\mathcal{B}(\mathcal{H}\right)$, where $p\in [1,2]$ and $R_{\Omega}(p)$ is the minimal positive root in $(0,1)$ of the equation
\begin{align*}
p\varphi_0(r)= 2\lambda_H \sum_{n=1}^{\infty}\varphi_n(r).
\end{align*} 
Based on the preceding discussion, it is clear that the research tradition concerning Bohr-type inequalities centers on refining and enhancing the classical Bohr's inequality. Thus, we aim to investigate an improved version of the generalized Bohr radius in \eqref{BS-eq-2.3} associated with $\varphi$, which requires redefining the definition as mentioned above.\vspace{1.2mm}
	
	\begin{defi} Let $ f\in H^{\infty}\left( \Omega, \mathcal{B} (\mathcal{H})\right) $ be given by $ f(z)=\sum_{n=m}^{\infty}A_nz^n $ in $ \mathbb{D}, $ $m=0,1,2,\dots$ with $ ||f(z)||_{H^{\infty}\left( \Omega, \mathcal{B} (\mathcal{H})\right) }\leq 1 $. For $ \varphi\in \mathcal{F} $, we define
		\begin{align*}
			&R^{\mu}_{p,q,m,\varphi}(f,\Omega,X)\\&:=\sup\biggl\{r\geq 0: ||A_m||^p\varphi_m(r)+\left(\sum_{n=m+1}^{\infty}||A_n||\varphi_n(r)+\mu(r)\mathcal{A}\left(f_m,\varphi,r\right)\right)^q\leq \varphi_{m}(r)\biggr\},
		\end{align*}
		where  $p\in(0,2]$ and $\mu:[0,1]\rightarrow[0,\infty)$ be a continuous function. The Bohr radius associated with $\varphi$ defined by
		\begin{align*}
			R^{\mu}_{p,q,m,\varphi}\left(\Omega,X\right):=\inf\biggl\{	R^{\mu}_{p,q,m,\varphi}(f,\Omega,X):||f(z)||_{H^{\infty}\left( \Omega, \mathcal{B} (\mathcal{H})\right) }\leq 1\biggr\}.
		\end{align*}
	\end{defi}
	It is clear that  $	R^{0}_{p,1,0,\varphi}\left(\Omega,\mathcal{B}(\mathcal{H})\right)=	R_{p,1,\varphi}\left(\Omega,\mathcal{B}(\mathcal{H})\right)$ for $p\in[1,2]$.\vspace{1.2mm}
	
In the following result, our goal is to establish a general version of the inequality \eqref{Eq-1.4} for the class of functions $f \in H^{\infty}(\Omega, \mathcal{B}(\mathcal{H}))$ and obtain the following result.
	\begin{thm}\label{BS-thm-4.9}
		Fix $m\in \mathbb{N}_0$ and $p\in(0,2].$ Let  $\mu:[0,1]\rightarrow[0,\infty)$ be a continuous function and $ f\in {H}^{\infty}\left(\Omega, \mathcal{B}(\mathcal{H})\right)$ be given by $ f(z)=\sum_{n=m}^{\infty}A_nz^n $ in $\mathbb{D}$  and $||f(z)||_{H^{\infty}\left( \Omega, \mathcal{B} (\mathcal{H})\right) }\leq 1,$ where $A_m=a_mI$ with $|a_m|<1$ and $A_n\in \mathcal{B}(\mathcal{H})$ for $n\geq m.$ If $\varphi=\{\varphi_n(r)\}^{\infty}_{n=m}\in \mathcal{F}$ satisfies the inequality    
		\begin{align}\label{BS-eq-4.10}
			p\varphi_m(r)> 2\lambda_H \sum_{n=m+1}^{\infty}\varphi_n(r)\;\;\mbox{ for}\;\; r\in [0, R_{\Omega}(p,m)),
		\end{align} 
		then the following sharp inequality holds:
		\begin{align*}
			\mathcal{N}^\mu_f\left(\varphi,p,m\right):=||A_m||^p\varphi_m(r)+\sum_{n=m+1}^{\infty}||A_n||\varphi_n(r)+\mu(r)\mathcal{A}_m\left(f_0,\varphi,r\right)\leq \varphi_m(r)
		\end{align*}
		for $|z|=r\leq R_{\Omega}(p,m),$ where $R_{\Omega}(p,m)$ is the minimal positive root in $(0,1)$ of the equation
		\begin{align}\label{BS-eq-2.11}
			p\varphi_m(r)= 2\lambda_H \sum_{n=m+1}^{\infty}\varphi_n(r).
		\end{align} 
		Then $R_{\Omega}(p,m)\leq 	R^{\mu}_{p,1,m,\varphi}\left(\Omega,\mathcal{B}(\mathcal{H})\right). $ That is $ R^{\mu}_{p,1,m,\varphi}\left(\Omega,\mathcal{B}(\mathcal{H})\right)>0 $ for $p\in (0,2].$
	\end{thm}


 Let $F:\mathbb{D}\rightarrow\mathbb{D}$ be analytic mapping. For $k\in \mathbb{N},$ we say that $z=0$ is a zero of order $k$ of $F$ if 
\begin{align*}
F(0)=F^{\prime}(0)=F^{\prime\prime}(0)=\dots=F^{k-1}(0)=0\;\;\mbox{but}\;\; F^{k}(0)\neq 0.
\end{align*}
An analytic mapping $\omega:\mathbb{D}\rightarrow\mathbb{D}$ with $\omega(0)=0$ is called a Schwarz mapping. We note that if $\omega$ is a Schwarz mapping such that $\omega(0)=0$ is a zero of order $k$ of $\omega,$ then the following estimation holds:
\begin{align}\label{Eq-2.7}
|\omega(z)|\leq|z|^k,\;\;\mbox{for}\;\;z\in \mathbb{D}.
\end{align}
In the next theorem, we obtain the Bohr-Rogosinski inequality for the class of functions $f\in H^{\infty}\left( \mathbb{D}, \mathcal{B} (\mathcal{H})\right)$ with help of $\{\varphi_n(r)\}^{\infty}_{n=0}\in \mathcal{F}.$
\begin{thm}\label{BS-thm-5.1}
For fixed $p\in (0,2],$ let $m\in \mathbb{N}_0$ and $N\in\mathbb{N}.$ Suppose that $\mu:[0,1)\rightarrow[0,\infty)$ be a continuous function and $\omega:\mathbb{D}\rightarrow\mathbb{D}$ with $\omega(0)=0$ be a Schwarz mapping of order $m$. Let $f\in  H^{\infty}\left( \mathbb{D}, \mathcal{B} (\mathcal{H})\right)$ be given by $f(z)=\sum_{n=0}^{\infty}A_nz^n$ in $\mathbb{D},$ where  $A_0=a_0I$ for $|a_0|<1,$ and $A_n\in \mathcal{B}(\mathcal{H})$ for all $n\in \mathbb{N}_0$ with $||f(z)||_{{H}^{\infty}\left(\mathbb{D},\mathcal{B}(\mathcal{H})\right)}\leq 1.$ If $\varphi=\{\varphi_n(r)\}^{\infty}_{n=0}\in \mathcal{F}$ satisfies the inequality 
\begin{align*}
p\left(\dfrac{1-r^m}{1+r^m}\right)\varphi_0(r)>2\mu(r)\sum_{n=N}^{\infty}\varphi_n(r),
\end{align*}
then the following sharp inequality holds:  
\begin{align*}
\mathcal{M}_f(r):=||f(\omega(z))||^{p}\varphi_0+\mu(r) \sum_{n=N}^{\infty}||A_n||\varphi_n(r)\leq \varphi_0(r)\;\;\mbox{for}\;\; r\leq R,
\end{align*}
where $R:=R_N(p,m,\mu)$ is the minimal positive root of the equation 
\begin{align*}
p\left(\dfrac{1-r^m}{1+r^m}\right)\varphi_0(r)-2\mu(r)\sum_{n=N}^{\infty}\varphi_n(r)=0.
\end{align*} 
In the case when 
\begin{align*}
p\left(\dfrac{1-r^m}{1+r^m}\right)\varphi_0(r)<2\mu(r)\sum_{n=N}^{\infty}\varphi_n(r),
\end{align*}
in some interval $(R,R+\epsilon),$ the number $R$ cannot be improved.
\end{thm}
We obtain the following list of results as the consequences of Theorem \ref{BS-thm-5.1}.
\begin{cor}
		Let $\varphi_n(r)=r^n$ for $n\in \mathbb{N}_0.$ If $m=1,$ $p\in (0,2],$  $\mu:[0,1)\rightarrow[0,\infty)$ be a continuous function and $f$ be as in Theorem \ref{BS-thm-5.1} , then 
		\begin{align*}
			||f(\omega(z))||^p+\mu(r)\sum_{n=N}^{\infty}||A_n||r^n\leq 1\;\;\mbox{for}\;\; |z|=r\leq R_N=R_N(p,m,\mu),
		\end{align*}
		where $R_N$ is the minimal root in $(0,1)$ of the equation $p(1-r)^2-2\mu(r)(1+r)r^N=0.$
		The constant  $R_N$ is best possible.	
	\end{cor}
	
	\begin{cor}
		Let $\varphi_n(r)=(n+1)r^n$ for $m,n\in \mathbb{N}_0.$  If $N=1,$ $p\in (0,2],$  $\mu:[0,1)\rightarrow[0,\infty)$ be a continuous function and $f$ be as in Theorem \ref{BS-thm-5.1} , then 
		\begin{align*}
			||f(\omega(z))||^p+\mu(r)\sum_{n=1}^{\infty}(n+1)||A_n||r^n\leq 1\;\;\mbox{for}\;\; |z|=r\leq R_1,
		\end{align*}
		where $R_1$ is the minimal root in $(0,1)$ of the equation $p\left(\frac{1-r^m}{1+r^m}\right)-2\mu(r) \frac{r(2-r)}{(1-r)^2}=0.$
		The constant  $R_1$ is best possible.	
	\end{cor}
	\begin{table}[ht]
		\centering
		\begin{tabular}{|l|l|l|l|l|l|l|l|l|l|}
			\hline
			$\;\;p$& $\;\;m$&$\;\;\mu$& $\;\;R_1$ \\
			\hline
			$0.5$& $1 $&$1 $& $0.090368$\\
			\hline
			$1$& $2 $&$3 $& $0.073469$\\
			\hline
			$1.5$& $5 $&$10 $& $0.067495$\\
			\hline
			$2$& $10 $&$100 $& $0.00496281$\\
			\hline
		\end{tabular}
		\vspace{2.5mm}
		\caption{This table exhibits the the approximate values of the roots $ R_1 $  for different values of $p,$ $m,$ $\mu$. }
		\label{tabel-1}
	\end{table}
	\begin{cor}
		Let $\varphi_0(r)=1$ and  $\varphi_n(r)=n^2r^n$ for $n\in \mathbb{N}.$  If $N=1,$ $p\in (0,2],$  $\mu:[0,1)\rightarrow[0,\infty)$ be a continuous function and $f$ be as in Theorem \ref{BS-thm-5.1} , then 
		\begin{align*}
			||f(\omega(z))||^p+\sum_{n=1}^{\infty}n^2||A_n||r^n\leq 1\;\;\mbox{for}\;\; |z|=r\leq R_2=R(p,m,\mu),
		\end{align*}
		where $R_2$ is the minimal root in $(0,1)$ of the equation 
		\begin{align*}
			p\left(\frac{1-r^m}{1+r^m}\right)-2\mu(r) \frac{r(1+r)}{(1-r)^3}=0.
		\end{align*}
		The constant  $R_N$ is the best possible.	
	\end{cor}
	
	\begin{table}[ht]
		\centering
		\begin{tabular}{|l|l|l|l|l|l|l|l|l|l|}
			\hline
			$\;\;p$& $\;\;m$&$\;\;\mu$& $\;\;R_2$ \\
			\hline
			$0.5$& $1 $&$1 $& $0.119726$\\
			\hline
			$1$& $5 $&$10 $& $0.0421611$\\
			\hline
			$1.5$& $10 $&$25 $& $0.026917$\\
			\hline
			$2$& $15 $&$30 $& $0.0295861$\\
			\hline
		\end{tabular}
		\vspace{2.5mm}
		\caption{This table exhibits the the approximate values of the roots $ R_2 $  for different values of $p,$ $m,$ $\mu$. }
		\label{tabel-2}
	\end{table}
	
	\begin{cor}
		Let $\varphi_{2n}(r)=r^{2n}$ and  $\varphi_{2n+1}(r)=0$ for $n\in \mathbb{N}_0.$  If $N=1,$ $p\in (0,2],$  $\mu:[0,1)\rightarrow[0,\infty)$ be a continuous function and $f$ be as in Theorem \ref{BS-thm-5.1} , then 
		\begin{align*}
			||f(\omega(z))||^p+\sum_{n=1}^{\infty}||A_n||r^{2n}\leq 1\;\;\mbox{for}\;\; |z|=r\leq R_3,
		\end{align*}
		where $R_3$ is the minimal root in $(0,1)$ of the equation 
		\begin{align*}
			p\left(\frac{1-r^m}{1+r^m}\right)-2\mu(r) \left(\frac{r^2}{1-r^2}\right)=0.
		\end{align*}
		The constant  $R_3$ is the best possible.	
	\end{cor}
	\begin{table}[ht]
		\centering
		\begin{tabular}{|l|l|l|l|l|l|l|l|l|l|}
			\hline
			$\;\;p$& $\;\;m$&$\;\;\mu$& $\;\;R_3$ \\
			\hline
			$0.5$& $1 $&$1 $& $0.333333$\\
			\hline
			$1$& $5 $&$10 $& $0.218115$\\
			\hline
			$1.5$& $10 $&$25 $& $0.170664$\\
			\hline
			$2$& $15 $&$30 $& $0.0295861$\\
			\hline
		\end{tabular}
		\vspace{2.5mm}
		\caption{This table exhibits the the approximate values of the roots $ R_3 $  for different values of $p,$ $m,$ $\mu$. }
		\label{tabel-3}
	\end{table}
	
	\begin{cor}
		Let $\varphi_{0}(r)=1,$ $\varphi_{2n}(r)=0$ and  $\varphi_{2n+1}(r)=r^{2n-1}$ for $n\in \mathbb{N}.$  If $N=1,$ $p\in (0,2],$  $\mu:[0,1)\rightarrow[0,\infty)$ be a continuous function and $f$ be as in Theorem \ref{BS-thm-5.1} , then 
		\begin{align*}
			||f(\omega(z))||^p+\mu(r)\sum_{n=1}^{\infty}||A_n||r^{2n-1}\leq 1\;\;\mbox{for}\;\; |z|=r\leq R_4,
		\end{align*}
		where $R_4$ is the minimal root in $(0,1)$ of the equation 
		\begin{align*}
			p\left(\frac{1-r^m}{1+r^m}\right)-2\mu(r) \left(\frac{r}{1-r^2}\right)=0.
		\end{align*}
		The constant  $R_4$ is the best possible.	
	\end{cor}
	
	\begin{table}[ht]
		\centering
		\begin{tabular}{|l|l|l|l|l|l|l|l|l|l|}
			\hline
			$\;\;p$& $\;\;m$&$\;\;\mu$& $\;\;R_4$ \\
			\hline
			$0.5$& $1 $&$1 $& $0.171573$\\
			\hline
			$1$& $5 $&$10 $& $0.049875$\\
			\hline
			$1.5$& $10 $&$25 $& $0.029973$\\
			\hline
			$2$& $15 $&$30 $& $0.0332964$\\
			\hline
		\end{tabular}
		\vspace{2.5mm}
		\caption{This table exhibits the the approximate values of the roots $ R_4 $  for different values of $p,$ $m,$ $\mu$. }
		\label{tabel-4}
	\end{table}
	
\subsection{\bf Proof of Theorems \ref{th-3.3}, \ref{th-3.4}, \ref{th-3.5}, \ref{BS-thm-4.9}, and \ref{BS-thm-5.1}}
	\begin{proof}[\bf Proof of Theorem \ref{th-3.3}]
		Suppose that $ f\in {H}^{\infty}\left(\Omega, \mathcal{B}(\mathcal{H})\right)$ be given by $ f(z)=\sum_{n=0}^{\infty}A_nz^n $ in $\mathbb{D} $ with $A_0=a_0I,$ $|a_0|<1.$ Then, by the virtue of \eqref{e-2.3}, we have
		\begin{equation}\label{e-4.12a}
			||A_n||\leq \lambda_H ||I-|A_0|^2||=\lambda_H||I-|a_0|^2I||=\lambda_H(1-|a_0|^2)\;\;\mbox{for}\;\; n\geq 1.
		\end{equation}
		A simple computation using \eqref{e-4.12a} shows that
		\begin{align}\label{e-3.4a}	S_r&=\sum_{n=1}^{\infty}n||A_n||^2r^{2n}\nonumber\leq \lambda_H^2(1-|a_0|^2)^2\sum_{n=1}^{\infty}nr^{2n}=\lambda_H^2(1-|a_0|^2)^2\left(\frac{r}{1-r^2}\right)^2.
		\end{align}
		\noindent For $ r\leq 1/(1+2\lambda_H) $, it is easy to see that
		\begin{equation}\label{e-4.12b}
			\frac{r}{1-r^2}\leq\frac{1+2\lambda_H}{4(1+\lambda_H)\lambda_H}.
		\end{equation}
		Thus, it follows from \eqref{e-4.12a} and \eqref{e-4.12b} that
		\begin{align*} 
			\sum_{n=0}^{\infty}||A_n||r^n+P\left(S_r\right)&= ||A_0||+\sum_{n=1}^{\infty}||A_n||r^n+\sum_{j=1}^{m}\left(\frac{1+\lambda_H}{1+2\lambda_H}\right)^{2j}\left(S_r\right)^j\\&=1-\frac{(1-|a_0|^2)}{16^m}\left(\frac{16^m}{1+|a_0|}-\frac{16^m}{2}-\sum_{j=1}^{m}16^{m-j}(1-|a_0|^2)^{2j-1}\right)\\& =1-\frac{(1-|a_0|^2)}{16^m} J_1(|a_0|),
		\end{align*}
		where 
		\begin{equation*}
			J_1(x):=\frac{16^m}{1+x}-\frac{16^m}{2}-\sum_{j=1}^{m}16^{m-j}(1-x^2)^{2j-1},\;\;\mbox{where}\;\;x\in [0,1].
		\end{equation*}
		Further, we see that
		\begin{equation*}
			J_1(0)=\frac{16^m}{2}-\frac{16^m-1}{15}=\frac{13}{30}16^m+\dfrac{1}{15}>0\;\;\mbox{and}\;\; J_1(1)=0.
		\end{equation*}
		A simple computation shows that 
\begin{align*}
	J_1^{\prime}(x)&=-\frac{16^m}{(1+x)^2}+2x\sum_{j=1}^{m}16^{m-j}(2j-1)(1-x^2)^{2j-2}\\& \leq-\frac{16^m}{4}+\dfrac{2\left(17(16)^m-30m-17\right)}{225}\\&=-\left(\dfrac{89(16)^m+4(30m+17)}{900}\right)<0.
\end{align*}
		Clearly, $ J_1(x) $ is a decreasing function in $ [0,1] $ and hence $ J_1(x)\geq J_1(1)=0 $ for $ x\in[0,1]. $ Thus, we have 
		\begin{equation*}
			\sum_{n=0}^{\infty}||A_n||r^n+P\left(S_r\right)\leq 1\;\; \mbox{for}\;\; r\leq\frac{1}{1+2\lambda_H}.
		\end{equation*}
	Evidently, the equality
		\begin{equation*}
			\sum_{n=0}^{\infty}||A_n||\left(\frac{1}{1+2\lambda_H}\right)^n+P\left(S_\frac{1}{1+2\lambda_H}\right)=1
		\end{equation*}
		holds for the function $ f $ of the form $ f(z)=\sum_{n=0}^{\infty} A_nz^n$
		if, and only if, $ f(z)=cI $ with $ |c|=1. $ This completes the proof.
	\end{proof}
	\begin{proof}[\bf Proof of Theorem \ref{th-3.4}]
		Suppose that $ f\in {H}^{\infty}\left(\Omega, \mathcal{B}(\mathcal{H})\right)$ be given by $ f(z)=\sum_{n=0}^{\infty}A_nz^n $ in $\mathbb{D} $ with $A_0=a_0I,$ $|a_0|<1.$ 
		Then for $ r\leq 1/(1+2\lambda_H) $ and in view of \eqref{e-4.12a}, an elementary computation gives that 
		\begin{align*}
			\mathcal{B}_{\lambda_H}^f(\beta,r) & \leq |a_0|+\bigg(\lambda_H(1-|a_0|^2)+\beta\lambda_H^2(1-|a_0|^2)^2\bigg)\frac{r}{1-r}\\&\leq |a_0|+\frac{1-|a_0|^2}{2}\left(1+\lambda_H\beta(1-|a_0|^2)\right)\\&= 1-\frac{1-|a_0|^2}{2}F_{1}(|a_0|),
		\end{align*}
		where 
		\begin{align*}
			F_{1}(x):=\dfrac{2}{(1+x)}-1-\lambda_H\beta(1-x^2).
		\end{align*}
		Since $ 0\leq\beta\leq 1/{4\lambda_H} $, it follows that $ F_{1}(0)=1-\lambda_H\beta\geq 0 $ and $ F_{1}(1)=0 $. 
		Our aim is to show that $ F_{1}(x)\geq 0 $ for $ x\in [0,1] $. We see that
		\begin{equation*}
			F^{\prime}_{1}(x)=-\frac{2}{(1+x)^2}+2\beta\lambda_H x\leq -\dfrac{1}{2}+2\beta\lambda_H\leq 0\;\; \mbox{for}\;\; 0\leq \beta\leq\frac{1}{4\lambda_H}.
		\end{equation*}
		Hence, $ F_{1} $ is a decreasing function of $x$ in $ [0,1] $ and hence we have $ F_{1}(x)\geq F_{1}(1)=0 $ for $ x\in [0,1] $.\vspace{1.2mm}
		
		Thus, we obtain
		\begin{equation*}
			\mathcal{B}_{\lambda_H}^f(\beta,r) \leq 1\;\;\mbox{for}\;\; r\leq r^*:=\frac{1}{1+2\lambda_H}.
		\end{equation*}
		It can be shown that the equality 
		\begin{equation*}
			\mathcal{B}_{\lambda_H}^f(\beta,r^*) =1
		\end{equation*}
		holds for the function $ f $ of the form $ f(z)=\sum_{n=0}^{\infty} A_nz^n$
		if, and only if, $ f(z)= cI $ with $ |c|=1. $ This completes the proof.
	\end{proof}
	
	\begin{proof}[\bf Proof of Theorem \ref{th-3.5}]
		Suppose that $ f\in {H}^{\infty}\left(\Omega, \mathcal{B}(\mathcal{H})\right)$ be given by $ f(z)=\sum_{n=0}^{\infty}A_nz^n $ in $\mathbb{D} $ with $A_0=a_0I,$ $|a_0|<1.$ For $ r\leq 1/(1+2\lambda_H) $ and in view of \eqref{e-4.12a}, a computation shows that
		\begin{align*} 
			&\mathcal{C}_{\lambda_H}^f(r)\\&\leq|a_0|+\lambda_H(1-|a_0|^2)\frac{r}{1-r}+\left(\frac{1+\lambda_H}{2\lambda_H(1+|a_0|)}+\frac{2}{3}\frac{(1+\lambda_H)r}{1-r}\right)\lambda_H^2(1-|a_0|^2)^2\frac{r^2}{1-r^2}\\&\leq |a_0|+\frac{1-|a_0|^2}{2}+\left(\frac{1+\lambda_H}{2\lambda_H(1+|a_0|)}+\frac{2}{3}\frac{(1+\lambda_H) r}{1-r}\right)\frac{\lambda_H}{4(1+\lambda_H)}(1-|a_0|)^2(1+|a_0|)^2\\&=1-\dfrac{(1-|a_0|)^2}{48}(24-(5+|a_0|)(1+|a_0|))\leq 1,
		\end{align*}
		because 
		\begin{equation*}
			24-(5+|a_0|)(1+|a_0|)\geq 0\;\; \mbox{for}\;\;|a_0|<1.
		\end{equation*}
		Therefore, we obtain the desired inequality 
		\begin{equation*}
			\mathcal{C}_{\lambda_H}^f(r)\leq 1\;\; \mbox{for}\;\; r\leq\frac{1}{1+2\lambda_H}.
		\end{equation*}
	It is a simple task to examine the equality 
		\begin{equation*}
			\mathcal{C}_{\lambda_H}^f(r)=1
		\end{equation*}
		holds for a function $ f\in {H}^{\infty}\left(\Omega, \mathcal{B}(\mathcal{H})\right) $ if, and only if, $ f(z)=cI $ with $ |c|=1 $.
		This completes the proof.
	\end{proof}
	
	\begin{proof}[\bf Proof of Theorem \ref{BS-thm-4.9}]
		Suppose that $ f\in {H}^{\infty}\left(\Omega, \mathcal{B}(\mathcal{H})\right)$ be given by $ f(z)=\sum_{n=m}^{\infty}A_nz^n $ in $\mathbb{D} $ with $A_m=a_mI,$ $|a_m|<1.$	We observe that $f(z)=z^mh(z),$ where $h: \Omega\rightarrow \mathcal{B}(\mathcal{H})$ is holomorphic function of the form $h(z)=\sum_{n=m}^{\infty} A_nz^{n-m}$ in $\mathbb{D}$ with $||h||_{{H}^{\infty}\left(\Omega, \mathcal{B}(\mathcal{H})\right)}\leq 1.$ In view of \eqref{e-4.12a}, we have
		\begin{align}\label{BS-eq-5.4}
			||A_n||\leq \lambda_H||I-|A_m|^2||= \lambda_H||I-|a_m|^2I||=\lambda_H(1-|a_m|^2);\;\;\mbox{for}\;\; n\geq m+1.
		\end{align}
		To proceed further in the proof, we will use the following inequality provided in \cite{Kayumov-Khammatova-Ponnusamy-MJM-2022}
		\begin{align}\label{eee-2.5}
			\dfrac{1-t^p}{1-t^2}\geq \dfrac{p}{2}\;\;\;\mbox{for\; all}\;\; t\in [0,1)\;\;\mbox{and}\;\; p\in(0,2].
		\end{align}
		
		In view of \eqref{BS-eq-5.4} and \eqref{eee-2.5},  a simple computation shows that 
		\begin{align*}
			\mathcal{N}^\mu_f\left(\varphi,p,m\right)&\leq |a_m|^p\varphi_m(r)+\lambda_H(1-|a_m|^2)\sum_{n=m+1}^{\infty}\varphi_m(r)\\&+\lambda^2_H\mu(r) (1-|a_m|^2)^2\sum_{n=m+1}^{\infty}\left(\dfrac{\varphi_{2n}(r)}{1+|a_m|}+\Phi_{2n+1}(r)\right)\\&=\varphi_m(r)+\lambda_H(1-|a_m|^2)\bigg(\sum_{n=m+1}^{\infty}\varphi_m(r)-\dfrac{(1-|a_m|^p)}{\lambda_H(1-|a_m|^2)}\varphi_m(r)\\&\quad+\lambda_H\mu(r)(1-|a_m|^2)\sum_{n=m+1}^{\infty}\left(\dfrac{\varphi_{2n}(r)}{1+|a_m|}+\Phi_{2n+1}(r)\right) \bigg)\\&\leq \varphi_m(r)+\lambda_H(1-|a_m|^2)\mathcal{N}^{*}(|a_m|,\varphi_m),
		\end{align*}
		where 
		\begin{align*}
			&\mathcal{N}^{*}(|a_m|,\varphi_m)\\&:=\sum_{n=m+1}^{\infty}\varphi_m(r)-\dfrac{p}{2\lambda_H}\varphi_m(r)+\lambda_H\mu(r)(1-|a_m|^2)\sum_{n=m+1}^{\infty}\left(\dfrac{\varphi_{2n}(r)}{1+|a_m|}+\Phi_{2n+1}(r)\right). 
		\end{align*}	
		
		In order to establish the desire inequality, it is sufficient to show $\mathcal{N}^{*}(|a_m|,\varphi_m)\leq 0$ for $|z|=r\leq R_{\Omega}(p).$ Choosing $|a_m|$ is very close to $1$ i.e. $|a_m|\rightarrow 1^{-}$ and using the given inequality \eqref{BS-eq-4.10}, we obtain 
		\begin{align*}
		\lim\limits_{|a_m|\rightarrow 1^-}	\mathcal{N}^{*}(|a_m|,\varphi_m)=\sum_{n=m+1}^{\infty}\varphi_n(r)-\dfrac{	p}{2\lambda_H}\varphi_m(r)\leq 0.
		\end{align*}
		Thus, the desired inequality   $\mathcal{N}^\mu_f\left(\varphi,p,m\right)\leq \varphi_m(r)$ is obtained for $r\leq R_\Omega(p),$ where $R_{\Omega}(p)$ is the minimal positive root in $(0,1)$ of the equation \eqref{BS-eq-2.11}. 
		This completes the proof.
	\end{proof}
 	The following lemma will play a significant roles in proving of Theorem  \ref{BS-thm-5.1}.
 	\begin{lemA}\emph{(\cite{Anderson- Rovnyak-Mathematika-2006})}\label{Lem-BS-4.1}
 		Let $B(z)$ be an analytic function with values in $\mathcal{B}(\mathcal{H})$ and $||f(z)||_{{H}^{\infty}\left(\mathbb{D},\mathcal{B}(\mathcal{H})\right)}\leq 1$ on $\mathbb{D}.$ Then
 		\begin{align*}
 			(1-|a|^2)^{n-1}\bigg|\bigg|\dfrac{B^{n}(a)}{n!}\bigg|\bigg|\leq \dfrac{||I-B^*(a)B(a)||^{1/2}||I-B(a)B^*(a)||^{1/2}}{1-|a|^2}
 		\end{align*}
 		for each $a\in \mathbb{D}$ and $n\in \mathbb{N}.$
 	\end{lemA}
 	Suppose that $f\in {H}^{\infty}\left(\mathbb{D} ,\mathcal{B}(\mathcal{H})\right)$ given by $f(z)=\sum_{n=0}^{\infty}A_nz^n$ in $\mathbb{D},$ where  $A_0=a_0I$ with $|a_0|<1,$ and $||f(z)||_{{H}^{\infty}\left(\mathbb{D},\mathcal{B}(\mathcal{H})\right)}\leq 1.$ Then, by the virtue of Lemma \ref{Lem-BS-4.1}, putting $a=0,$ we obtain (see \cite{Allu-Halder-CMB-2022})
 	\begin{align}\label{BS-eqq-4.1}
 		||A_n||\leq \bigg|\bigg|I-|A_0|^2\bigg|\bigg|=1-|a_0|^2\;\;\mbox{for}\;\; n\geq1.
 	\end{align}
 	
 		\begin{proof}[\bf Proof of Theorem \ref{BS-thm-5.1}]
 			Suppose that $f\in {H}^{\infty}\left(\mathbb{D},\mathcal{B}(\mathcal{H})\right)$ given by $f(z)=\sum_{n=0}^{\infty}A_nz^n$ in $\mathbb{D}$ with $A_n\in\mathcal{B}(\mathcal{H}),$ $n\in \mathbb{N}$ and $A_0=a_0I.$  Let $||A_0||=|a_0|:=a.$ Then we have (see \cite{Allu-Halder-CMB-2022})
 			\begin{align}\label{BS-eq-3.12}
 				||f(z)||\leq \dfrac{||A_0||+|z|}{1+||A_0|||z|},\;\;\; z\in \mathbb{D} . 
 			\end{align}

 			Since $t\rightarrow \dfrac{t+a}{1+ta}$ is increasing function of $t\in [0,1),$   by the virtue of \eqref{BS-eq-3.12} and \eqref{Eq-2.7}, we have 
 			\begin{align}\label{BS-eq-4.1}
 				||f(\omega(z))||\leq \dfrac{a+r^m}{1+ar^m},\;\;\mbox{|z|=r<1.}
 			\end{align}
 			By the virtue of the inequalities $\eqref{BS-eqq-4.1}$ and \eqref{BS-eq-4.1}, we obtain
 			\begin{align}\label{BS-eq-4.2}
 				&\mathcal{M}_f(r)\leq \left( \dfrac{a+r^m}{1+ar^m}\right)^p+(1-a^2)\mu(r)\sum_{n=N}^{\infty}\varphi_n(r) =\varphi_0(r)+\mathcal{T}_{p,m,N,\mu}(a),
 			\end{align}
 			where \begin{align}
 				\mathcal{T}_{p,m,N,\mu}(a):=\left(\left( \dfrac{a+r^m}{1+ar^m}\right)^p-1\right)\varphi_0(r)+(1-a^2)\mu(r)\sum_{s=N}^{\infty}\varphi_n(r).
 			\end{align}
 			To prove the desired inequality, it is sufficient to show that $\mathcal{T}_{p,m,N,\mu}(a)\leq 0$ for $r\leq R$ and $a\in [0,1).$ We see that $\mathcal{T}_{p,m,N,\mu}(1)= 0.$ \vspace{1.2mm}
 			
 			First, we show that $\mathcal{T}_{p,m,N,\mu}(a)$ is an increasing function of $a\in [0,1),$ whenever $0<p\leq 1.$ Indeed, a direct computation shows that
 			\begin{align*}
 				\mathcal{T}^{\prime}_{p,m,N,\mu}(a)=p(1-r^{2m})\dfrac{(a+r^m)^{p-1}}{(1+ar^m)^{p+1}}\varphi_0(r)-2\mu(r) a\sum_{n=N}^{\infty}\varphi_n(r),
 			\end{align*}
 			\begin{align*}
 				\mathcal{T}^{\prime \prime}_{p,m,N,\mu}(a)=p(1-r^{2m})\dfrac{(a+r^m)^{p-2}}{(1+ar^m)^{p+2}}\left(p-1-2ar^m-(p+1)r^{2m}\right)\varphi_0(r)-2\mu(r) \sum_{n=N}^{\infty}\varphi_n(r).
 			\end{align*}
 			Obviously, $\mathcal{T}^{\prime \prime}_{p,m,N,\mu}(n)\leq 0$ for all $a\in [0,1],$ whenever $0<p\leq 1.$ Hence 
 			\begin{align*}
 				\mathcal{T}^{\prime}_{p,m,N,\mu}(a)\geq \mathcal{T}^{\prime}_{p,m,N,\mu}(1)=p\left(\dfrac{1-r^m}{1+r^m}\right)\varphi_0(r)-2\mu(r) \sum_{n=N}^{\infty}\varphi_n(r)\geq 0,
 			\end{align*}
 			by assumption for $r\leq R.$ Thus, for $r\leq R$ and $0<p\leq 1,$ $\mathcal{T}_{p,m,N,\mu}(a)$ is an increasing function of $a\in [0,1]$ which in turn implies that $\mathcal{T}_{p,m,N,\mu}(a)\leq \mathcal{T}_{p,m,N,\mu}(1)=0$  for all $a\in [0,1]$ and the desired inequality follows.\vspace{1.2mm}
 			
 			Next, we show that $\mathcal{T}_{p,m,N,\mu}(a)\leq 0$ whenever $1<p\leq 2.$ By the similar argument of proof of \cite{Liu-Ponnusamy-PAMS-2021}, we find that $\Psi(\sqrt[m]{r})\geq a^{p-1}$ for all $r\in [0,1),$ where \begin{align*}
 				\Psi(r):=(1+r^m)^2\dfrac{(a+r^m)^{p-1}}{(1+ar^m)^{p+1}}.
 			\end{align*}
 			Since $0\leq a^{p-2}\leq 1$ and $1<p\leq 2,$ we may rewrite 
 			\begin{align*}
 				\mathcal{T}^{\prime}_{p,m,N,\mu}(a)&=p\dfrac{1-r^m}{1+r^m}\Psi(r)\varphi_0(r)-2\mu(r) a\sum_{n=N}^{\infty}\varphi_n(r)
 				\\&\geq a^{p-1}\left(p\left(\dfrac{1-r^m}{1+r^m}\right)\varphi_0(r)-2\mu(r) a^{2-p}\sum_{n=N}^{\infty}\varphi_n(r)\right)\\&\geq a^{p-1}\left(p\left(\dfrac{1-r^m}{1+r^m}\right)\varphi_0(r)-2\mu(r) \sum_{n=N}^{\infty}\varphi_n(r)\right)\geq 0
 			\end{align*}
 			for all $n\in [0,1]$, by assumption.  Again, $\mathcal{T}_{p,m,N,\mu}(a)$ is an increasing function of $a\in [0,1]$ whenever $1<p\leq 2,$ which in turn implies that $\mathcal{T}_{p,m,N,\mu}(a)\leq \mathcal{T}_{p,m,N,\mu}(1)=0$ for all $a\in [0,1].$ Thus, the desired inequality holds for $r\leq R.$\vspace{1.2mm}
 			
 			The second part of the proof is to show that the constant $R$ is sharp. Hence, we consider the function
 			\begin{align*}
 				f_a(z)=\left(\dfrac{a+z}{1+az}\right)I=A_0+\sum_{n=1}^{\infty}A_nz^n\;\mbox{for}\; z\in \mathbb{D},
 			\end{align*}
 			where $A_0=aI, $ $A_n=-(1-a^2)a^{n-1}I$ for some $a\in [0,1).$ For the function $	f_a,$ $\omega(z)=z^m,$ and $z=r,$ a straightforward computation yields that
 			\begin{align*}
 				\mathcal{M}_{f_a}(r)&={||f_a(r^m)||}^p\varphi_0(r)+\mu(r)\sum_{n=N}^{\infty}||A_n||\varphi_n(r)\\&= \varphi_0(r)+\left(\left(\dfrac{a+r^m}{1+ar^m}\right)^p-1\right)	\varphi_0(r)+\mu(r) (1-a^2)\sum_{n=N}^{\infty}a^{n-1}\varphi_n(r)\\&=\varphi_0(r)+(1-a)G_{\mu}(a,r),
 			\end{align*}
 			where 
 			\begin{align*}
 				G_{\mu}(a,r):=\dfrac{1}{(1-a)}\left(\left(\dfrac{a+r^m}{1+ar^m}\right)^p-1\right)\varphi_0(r)+\mu(r) (1+a)\sum_{n=N}^{\infty}a^{n-1}\varphi_n(r).
 			\end{align*}
 			For $r> R$ and $a$ sufficiently close to $1^{-}$, we see that
 			\begin{align*}
 				\lim\limits_{a\rightarrow 1^{-}}G_{\varphi}(a,r)=2\mu(r)\sum_{n=N}^{\infty}\varphi_n(r)-p\left(\dfrac{1-r^m}{1+r^m}\right)\varphi_0(r)>0,
 			\end{align*}
 			by the assumption. This verifies that the radius $R$ is best possible. This completes the proof.
 		\end{proof}
 		
 		\section{\bf Generalization of the Bohr inequality for the class $\mathcal{B}\left(\Omega_{\gamma}\right)$ }
 		
 		  In \cite{Kumar-CVEE-2022}, a generalized of the Bohr inequality for analytic functions $f$ studied on the domain $\Omega_{\gamma}$ with help of a sequence $ \{\varphi_n(r)\}_{n=0}^{\infty}\in \mathcal{F}$. The Bohr inequality has been studied (see \cite{Allu-Halder-PEMS-2023,Allu-Halder-CMB-2023}) for Banach spaces on the domain $ \Omega_{\gamma} $ with help of a sequence $ \{\varphi_n(r)\}_{n=0}^{\infty} \in \mathcal{F}.$ In this follow, sharp coefficient bound of the bounded analytic functions $f:\Omega_{\gamma}\rightarrow \mathcal{B}(\mathcal{H})$ with expansion $ f(z)=\sum_{n=0}^{\infty}A_nz^n $ in $\mathbb{D} $ such that $A_n\in \mathcal{B}(\mathcal{H}) $ for all $n\in \mathbb{N}_0$ was obtained and this bound will be instrumental in discussion.
 		
 		\begin{lemB}\emph{(see \cite[Lemma 1.42]{Allu-Halder-PEMS-2023})}\label{lem-2.1}
 			Let $f:\Omega_{\gamma}\rightarrow \mathcal{B}(\mathcal{H})$ be a bounded analytic function with expansion $ f(z)=\sum_{n=0}^{\infty}A_nz^n $ in $\mathbb{D} $ such that $A_n\in \mathcal{B}(\mathcal{H}) $ for all $n\in \mathbb{N}_{0}$ and $A_0$ is normal.  Then 
 			\begin{equation*}
 				||A_n||\leq\frac{||I-|A_0|^2||}{1+\gamma}\quad\mbox{for}\quad n\geq 1.
 			\end{equation*}
 		\end{lemB}
 	In \cite{Allu-Halder-PEMS-2023}, the authors have established a result which is a generalization of Theorem C by utilizing $\{\varphi_n(r)\}^{\infty}_{n=0}\in \mathcal{F}$ and obtained the following result.
 		
 		\begin{thmD}\emph{(see \cite[Theorem 1.4]{Allu-Halder-PEMS-2023})}\label{BS-thm-3.11} For fixed $p\in [1,2].$ Let $ f\in {H}^{\infty}\left(\Omega_{\gamma}, \mathcal{B}(\mathcal{H})\right)$ be given by $ f(z)=\sum_{n=0}^{\infty}A_nz^n $ in $\mathbb{D}$  and $||f(z)||_{H^{\infty}\left( \Omega_{\gamma}, \mathcal{B} (\mathcal{H})\right) }\leq 1,$ where $A_0=a_0I$ with $|a_0|<1$ and $A_n\in \mathcal{B}(\mathcal{H})$ for $n\in \mathbb{N}_0.$ If $\varphi=\{\varphi_n(r)\}^{\infty}_{n=0}\in  \mathcal{F} $ satisfies the inequality    
 			\begin{align}\label{BS-eq-3.6}
 				(1+\gamma)\varphi_0(r)>\dfrac{2}{p}\sum_{n=1}^{\infty}\varphi_n(r)\;\;\mbox{for}\;\; r\in [0,R_{\Omega_{\gamma}}(p)],
 			\end{align}
 			then the following sharp inequality holds:
 			\begin{align*}
 				\varphi_0(r)||A_0||^p+\sum_{n=1}^{\infty}||A_n||\varphi_n(r)\leq \varphi_0(r)
 			\end{align*}
 			for $|z|=r\leq R_{\Omega_{\gamma}}(p),$ where $R_{\Omega_{\gamma}}(p)$ is the minimal positive root in $(0,1)$ of the equation  
 			\begin{align}\label{BS-eq-3.2}
 				\dfrac{2}{p}\sum_{n=1}^{\infty}\varphi_n(r)=(1+\gamma)\varphi_0(r).
 			\end{align}
 			In the case when 
 			\begin{align*}
 				(1+\gamma)\varphi_0(r)<\dfrac{2}{p}\sum_{n=1}^{\infty}\varphi_n(r)
 			\end{align*}
 			in some interval $( R_{\Omega_{\gamma}}(p), R_{\Omega_{\gamma}}(p)+\epsilon)$, the number $ R_{\Omega_{\gamma}}(p)$ cannot be improved. If the function $\varphi_n(t)$ $(n\geq 0) $ are smooth, then the last condition is equivalent to the inequality 
 			\begin{align}\label{BS-eqq-2.3}
 				(1+\gamma)\varphi^{\prime}_0(t)<\dfrac{2}{p}\sum_{n=1}^{\infty}\varphi^{\prime}_n(t).
 			\end{align}
 		\end{thmD}
 		In the study of Bohr phenomenon, one aspect involves the examination of refined Bohr inequalities and proof of their sharpness for certain classes of functions (see \cite{Liu-Ponnusamy-PAMS-2021,Liu-Liu-Ponnusamy-BDSM-2021}). However, refining a Bohr inequality often leads to a different Bohr radius for different classes, hence, there is a curiosity in establishing a refined Bohr inequality for the same class without altering the radius. In this section, we are interested in obtaining a refined version of Theorem D. In light of this, it is natural to raise the following question.
 		\begin{ques}\label{BS-prob-2.1}
 			Is it possible to establish a sharp improved version of  Theorem D?
 		\end{ques}
 		We obtain  the following result as a refined version of Theorem D answering the Question \ref{BS-prob-2.1} completely. In fact, it worth noticing that the radius in Theorem \ref{BS-thm-3.1} is independent of the constant $ \mu $. In particular, when $ \mu=0 $ and $m=0,$ then we see that Theorem D is obtained from Theorem \ref{BS-thm-3.1}.\vspace{2mm}
 		
 \begin{thm}\label{BS-thm-3.1}
 For fixed $p\in (0,2]$ and $\mu:[0,1]\rightarrow[0,\infty)$ be a continuous function. Let $ f\in {H}^{\infty}\left(\Omega_{\gamma}, \mathcal{B}(\mathcal{H})\right)$ be given by $ f(z)=\sum_{n=m}^{\infty}A_nz^n $ in $\mathbb{D}$  and $||f(z)||_{H^{\infty}\left( \Omega_{\gamma}, \mathcal{B} (\mathcal{H})\right) }\leq 1,$ where $A_m=a_mI$ with $|a_m|<1$ and $A_n\in \mathcal{B}(\mathcal{H})$ for $n\geq m+1.$ If $\varphi=\{\varphi_n(r)\}^{\infty}_{n=m}\in \mathcal{F}$ satisfies the inequality 
 \begin{align}\label{Eq-BS-3.4}
 \varphi_m(r)>\dfrac{2}{p(1+\gamma)}\sum_{n=m+1}^{\infty}\varphi_{n}(r)\;\;\mbox{for}\;\;r\in [0,R_{\Omega_{\gamma}}(p,m)),
 \end{align}
 then the following sharp inequality holds:
 \begin{align*}
 \mathcal{M}^\mu_f\left(\varphi,m,p,\gamma\right):=\varphi_m(r)||A_m||^p+\sum_{n=m+1}^{\infty}||A_n||\varphi_n(r)+\mu(r)\mathcal{A}\left(f_m,\varphi,r\right)\leq \varphi_m(r)
 \end{align*}
 for $|z|=r\leq R_{\Omega_{\gamma}}(p,m),$ where $R_{\Omega_{\gamma}}(p,m)$ is the minimal positive root in $(0,1)$ of the equation
 \begin{align}\label{Eq-BS-3.5}
 	 \varphi_m(r)=\dfrac{2}{p(1+\gamma)}\sum_{n=m+1}^{\infty}\varphi_{n}(r).
 \end{align} 
 In the case when 
 \begin{align}\label{Eq-BS-3.6}
 (1+\gamma)\varphi_m(r)<\dfrac{2}{p}\sum_{n=m+1}^{\infty}\varphi_n(r)
 \end{align}
 in some interval $(R_{\Omega_{\gamma}}(p,m),R_{\Omega_{\gamma}}(p,m)+\epsilon)$, the number $R_{\Omega_{\gamma}}(p,m)$  cannot be improved. 
 \end{thm}
 		To serve our purpose, we consider a polynomial $Q(w)$ of degree $ m $ as follows
 		\begin{equation}\label{e-1.8}
 			Q(w):=c_1w+c_2w^2+\cdots+c_mw^m\;\; \mbox{for}\;\; c_j\in\mathbb{R}^{+},\;\; j=1, 2, \cdots, m
 		\end{equation}
 		 and obtain the following result.
 		
 		\begin{thm}\label{th-3.6}
 			For $ 0\leq\gamma<1, $ let $ f\in H^{\infty}\left( \Omega_{\gamma}, \mathcal{B} (\mathcal{H})\right) $ be given by $ f(z)=\sum_{n=0}^{\infty}A_nz^n $ in $ \mathbb{D} $ with $||f(z)||_{H^{\infty}\left( \Omega_{\gamma}, \mathcal{B} (\mathcal{H})\right) }\leq 1,$ where $A_0=a_0I$ with $|a_0|<1.$ Then 
 			\begin{equation*}
 				\sum_{n=0}^{\infty}||A_n||r^n+Q\left(S_{r(\gamma-1)}\right)\leq 1\;\;\mbox{for}\;\; r\leq r_0=\frac{1+\gamma}{3+\gamma},
 			\end{equation*} where the coefficients of $ Q(w) $ satisfies 	\begin{align}\label{BS-eq-2.6}
 				8c_1\left(\dfrac{3}{8}\right)^2+6c_2d_2\left(\dfrac{3}{8}\right)^4+\dots+2(2m-1)c_md_m\left(\dfrac{3}{8}\right)^{2m}= 1,
 			\end{align}
 			where
 			\[
 			d_s:=\max_{a\in [0,1]}\left(a(1+a)^2(1-a^2)^{2s-2}\right),
 			\quad s=2,3,\cdots, m.
 			\]
 			Furthermore, the quantities $ c_1, c_2, \ldots,c_m $ and $ (1+\gamma)/(3+\gamma) $ cannot be improved.
 		\end{thm}
 		The following lemma play a significant role in proving the Theorem \ref{th-3.6}.
 		\begin{lem}\label{lem-3.6}
 		If $ g : \mathbb{D}\rightarrow\overline{\mathbb{D}} $ be an analytic function and $ \gamma\in\mathbb{D} $ be such that 
 			\begin{align*}
 				g(z)=\sum_{n=0}^{\infty}A_n(z-\gamma)^n\,\,\mbox{for}\;\; |z-\gamma|<1-|\gamma|,
 			\end{align*}
 			then
 			\begin{align}\label{ee-2.3}
 				\sum_{n=0}^{\infty}||A_n||\rho^n+Q\left(S_{\rho}\right)\leq 1\;\; \text{for}\;\; \rho\leq \rho_0=\frac{1-|\gamma|^2}{3+|\gamma|},
 			\end{align}
 			where $ S_{\rho} $ denotes the area of the image of the disk $ \mathbb{D}(\gamma;\rho(1-|\gamma|)) $ under the mapping  $ g $ and the non-negative real coefficients $ c_1, c_2, \cdots, c_m $ of the polynomial $ Q(w) $ given by \eqref{e-1.8} satisfy
 			\begin{align*}
 				8c_1\left(\dfrac{3}{8}\right)^2+6c_2d_2\left(\dfrac{3}{8}\right)^4+\dots+2(2m-1)c_md_m\left(\dfrac{3}{8}\right)^{2m}= 1,
 			\end{align*}
 			where
 			\[
 			d_s:=\max_{a\in [0,1]}\left(a(1+a)^2(1-a^2)^{2s-2}\right),
 			\quad s=2,3,\cdots, m.
 			\]
 		\end{lem}
 			
 		\subsection{\bf Application of Theorem \ref{BS-thm-3.1}}
 		As application of Theorem \ref{BS-thm-3.1}, the following results are the counterparts of refined version of Bohr's theorem for the class $H^{\infty}\left(\Omega_{\gamma}, \mathcal{B} (\mathcal{H})\right)$. 
 		\begin{itemize}
 			\item [(i)] Let $\varphi_n(r)=r^n$, and $p=1, 2$, $ \mu=1 $ in  Theorem \ref{BS-thm-3.1}.  We obtain the following inequalities which are refined version of the Bohr inequality for the class $H^{\infty}\left( \Omega_{\gamma}, \mathcal{B} (\mathcal{H})\right)$ as follows:
 			\begin{align*}
 				\sum_{n=0}^{\infty}||A_n||r^n+	\mathcal{A}\left(f_0,r\right)\leq 1\;\;\mbox{ for}\;\; r\leq R_{\gamma}(1)= \dfrac{1+\gamma}{3+\gamma}.
 			\end{align*}
 			Moreover,
 			\begin{align*}
 				||A_0||^2+\sum_{n=1}^{\infty}||A_n||r^n+	\mathcal{A}\left(f_0,r\right)\leq 1\;\;\mbox{ for}\;\; r\leq R_{\gamma}(2)=\dfrac{1+\gamma}{2+\gamma}.
 			\end{align*}
 			Both the radii $ R_{\gamma}(1) $ and $ R_{\gamma}(2) $ are best possible. \vspace{2mm}
 			
 			\item [(ii)] Let $\varphi_{2n}(r)=r^{2n}$ and  $\varphi_{2n+1}(r)=0,$ $n\geq 0$, in  Theorem \ref{BS-thm-3.1}. We see that 
 			\begin{align*}
 				\mathcal{A}\left(f_0,\varphi,r\right)=\frac{1}{1+||A_0||}\sum_{n=1}^{\infty}||A_n||^{2n}r^{2n}
 			\end{align*}
 			and the refined inequality 
 			\begin{align*}
 				||A_0||^p+\sum_{n=1}^{\infty}||A_{2n}||r^{2n}+\frac{\mu(r)}{1+|a_0|}\sum_{n=1}^{\infty}|a_n|^{2n}r^{2n}\leq 1
 			\end{align*}
 			holds for $r\leq R^{\gamma}_2(p)$, where $R^{\gamma}_2(p)=\sqrt{\frac{p(1+\gamma)}{2+p(1+\gamma)}}$. The radius $R^{\gamma}_2(p)<1$ is best possible. \vspace{2mm}
 			
 			\item [(iii)] Let  $\varphi_0(r)=1$,  $\varphi_{2n}(r)=0$ and  $\varphi_{2n-1}(r)=r^{2n-1}, $ $n\geq 0$ in  Theorem \ref{BS-thm-3.1}. Then we have
 			\begin{align*}
 				\mathcal{A}\left(f_0,\varphi,r\right)=\sum_{n=1}^{\infty}||A_{2n-1}||^{4n-2}\frac{r^{4n+1}}{1-r^2}
 			\end{align*}
 			and we see that
 			\begin{align*}
 				||A_0||^p+\sum_{n=1}^{\infty}||A_{2n-1}||r^{2n-1}+\mu(r)\sum_{n=1}^{\infty}||A_{2n-1}||^{4n-2}\frac{r^{4n+1}}{1-r^2}\leq 1 \;\;\mbox{for}\;\;r\leq R^{\gamma}_3(p),
 			\end{align*}
 			where $R^{\gamma}_3(p)=\frac{\sqrt{1+p^2(1+\gamma)}-1}{p(1+\gamma)}$. The radius $R^{\gamma}_3(p)(<1)$ (independent of $\mu$) is best possible.		
 		\end{itemize}
 		\begin{proof}[\bf{Proof of Theorem \ref{BS-thm-3.1}}]
 			 Let $ f\in H^{\infty}\left( \Omega_{\gamma}, \mathcal{B} (\mathcal{H})\right) $ be given by $ f(z)=\sum_{n=m}^{\infty}A_nz^n $ in $ \mathbb{D} $ with $||f(z)||_{H^{\infty}\left( \Omega_{\gamma}, \mathcal{B} (\mathcal{H})\right) }\leq 1.$ We observe that $f(z)=z^mh(z),$ where $h:\Omega_{\gamma}\rightarrow \mathcal{B}(\mathcal{H})$ is analytic function of the form $h(z)=\sum_{n=m}^{\infty}A_nz^{n-m}$ in $\mathbb{D}$ with $||h(z)||_{H^{\infty}\left( \Omega_{\gamma}, \mathcal{B} (\mathcal{H})\right) }\leq 1.$  Let  $a=||A_m||<1.$ Then using the Lemma \ref{lem-2.1}, we obtain
 			\begin{align*}
 				&\mathcal{M}^\mu_f\left(\varphi,m,p,\gamma\right)\\&\leq a^p\varphi_m(r)+\dfrac{(1-a^2)}{1+\gamma}\sum_{n=m+1}^{\infty}\varphi_n(r)+\mu(r)\left(\dfrac{1-a^2}{1+\gamma}\right)^2\sum_{n=m+1}^{\infty}\left(\dfrac{\varphi_{2n}(r)}{1+a}+\Phi_{2n+1}(r)\right)\\&=\varphi_m(r)+\dfrac{(1-a^2)}{1+\gamma}\bigg(\sum_{n=m+1}^{\infty}\varphi_n(r)+\mu(r)\left(\dfrac{1-a^2}{1+\gamma}\right)\sum_{n=m+1}^{\infty}\left(\dfrac{\varphi_{2n}(r)}{1+a}+\Phi_{2n+1}(r)\right)\\&\quad-\left(\dfrac{1-a^p}{1-a^2}\right)(1+\gamma)\varphi_m(r)\bigg).
 			\end{align*} 
 			
 			Using the inequality \eqref{eee-2.5}, we see that
 			\begin{align*}
 			&\mathcal{M}^\mu_f\left(\varphi,m,p,\gamma\right)\\&\leq \varphi_m(r)+\dfrac{(1-a^2)}{1+\gamma}\bigg(\sum_{n=m+1}^{\infty}\varphi_n(r)+\mu(r)\left(\dfrac{1-a^2}{1+\gamma}\right)\sum_{n=m+1}^{\infty}\left(\dfrac{\varphi_{2n}(r)}{1+a}+\Phi_{2n+1}(r)\right)\\&\quad-\left(\frac{p}{2}\right)(1+\gamma)\varphi_m(r)\bigg)\\& =\varphi_m(r)+\dfrac{(1-a^2)}{1+\gamma}H_{\gamma, \mu}(a,p,m),
 			\end{align*}
 			where
 			\begin{align*}
 				H_{\gamma, \mu}(a,p,m):&=\sum_{n=m+1}^{\infty}\varphi_n(r)+\mu(r)\left(\dfrac{1-a^2}{1+\gamma}\right)\sum_{n=m+1}^{\infty}\left(\dfrac{\varphi_{2n}(r)}{1+a}+\Phi_{2n+1}(r)\right)\\&\quad-\dfrac{p(1+\gamma)}{2}\varphi_m(r).
 			\end{align*}
 			In order to establish the desire inequality, it is sufficient to show $H_{\gamma, \mu}(a,p,m)\leq 0$ for $|z|=r\leq R_\gamma(p,m).$ Choosing $a$ is very close to $1$ i.e. $a\rightarrow 1^{-}$ and using the given inequality \eqref{Eq-BS-3.4}, we obtain 
 			\begin{align*}
 				\lim\limits_{a\rightarrow 1^-}H_{\gamma, \mu}(a,p,m)=\sum_{n=m+1}^{\infty}\varphi_n(r)-\dfrac{	p(1+\gamma)}{2}\varphi_m(r)\leq 0.
 			\end{align*}
 			Thus, the desired inequality   $\mathcal{M}^\mu_f\left(\varphi,m,p,\gamma\right)\leq \varphi_m(r)$ is obtained for $r\leq R_\gamma(p,m).$\vspace{1.1mm}
 			
 			Next part of the proof is show that the radius $R_\gamma(p,m)$ is best possible. Henceforth, we consider function 
 			\begin{align*}
 				F_{a,m}(z):=z^mh_a(z)\;\;\mbox{for}\;\; z\in \Omega_{\gamma}\;\mbox{and}\; a\in (0,1),
 			\end{align*} 
 			where 
 			\begin{align}\label{eeee-2.6} 
 				h_a(z):=\left(\dfrac{a-\gamma-(1-\gamma)z}{1-a\gamma-a(1-\gamma)}\right)I\;\;\mbox{for}\;\; z\in \Omega_{\gamma}\;\mbox{and}\; a\in (0,1).
 			\end{align}
 			We define two functions $ H_1 $ and $ H_2 $ by $H_1:\mathbb{D}\rightarrow\mathbb{D}$ by $H_1(z)=(a-z)/(1-az)$ and $H_2:\Omega_{\gamma}\rightarrow\mathbb{D}$ by $H_2(z)=(1-\gamma)z+\gamma.$ Then, it is easy to see that the function $h_a=H_1\circ H_2$ maps $ \Omega_{\gamma} $ univalently onto $\mathbb{D}.$ We note that $h_a$ is analytic in $\Omega_{\gamma}$ and $||h_a||\leq1$ and hence $||F_{a,m}(z)||\leq1.$ An easy computation yields that 
 			\begin{align*}
 				h_a(z)=\left(\dfrac{a-\gamma-(1-\gamma)z}{1-a\gamma-a(1-\gamma)}\right)I=A_0-\sum_{n=1}^{\infty}A_nz^n\;\;\mbox{in}\;\;\mathbb{D}\;\; \mbox{and}\;\;a\in (0,1),
 			\end{align*}   
 			where
 			\begin{align}\label{eee-2.7}
 				A_0:=\left(\dfrac{a-\gamma}{1-a\gamma}\right)I\;\mbox{and}\; A_n:=\left(\dfrac{1-a^2}{a(1-a\gamma)}\left(\dfrac{a(1-\gamma)}{1-a\gamma}\right)^n\right)I\;\;\;\mbox{for}\; n\in \mathbb{N}.
 			\end{align}
 			Therefore, we have $F_{a,m}(z)=z^mh_a(z)=A_0z^m-\sum_{n=m+1}^{\infty}A_nz^n$ in $\mathbb{D}.$
 			In view of the power series expansion of $F_{a,m},$ a routine computation gives that
 			\begin{align*}
 				&\mathcal{M}^\mu_{h_a}\left(\varphi,m,p,\gamma\right)\\&=\left(\dfrac{a-\gamma}{1-a\gamma}\right)^p\varphi_m(r)+\sum_{n=m+1}^{\infty}\dfrac{1-a^2}{a(1-a\gamma)}\left(\dfrac{a(1-\gamma)}{1-a\gamma}\right)^n\varphi_n(r)\\&\quad+\mu(r)\sum_{n=m+1}^{\infty}\left(\dfrac{1-a^2}{a(1-a\gamma)}\right)^2\left(\dfrac{a(1-\gamma)}{1-a\gamma}\right)^{2n}\left(\dfrac{\varphi_{2n}(r)}{1+\left(\dfrac{a-\gamma}{1-a\gamma}\right)}+\Phi_{2n+1}(r)\right)\\&=\varphi_m(r)+\dfrac{(1-a)}{1+\gamma}\left(2\sum_{n=m+1}^{\infty}\varphi_n(r)-p(1+\gamma)\varphi_m(r)\right)\\&\quad+(1-a)\left(\sum_{n=m+1}^{\infty}\dfrac{(1+a)}{a(1-a\gamma)}\left(\dfrac{a(1-\gamma)}{1-a\gamma}\right)^n\varphi_n(r)-\dfrac{2}{1-\gamma}\sum_{n=m+1}^{\infty}\varphi_n(r)\right)\\&\quad+(1-a)\left(p\left(\dfrac{1+\gamma}{1-\gamma}\right)+\frac{1}{1-a}\bigg(\left(\dfrac{a-\gamma}{1-a\gamma}\right)^p-1\bigg)\right)\varphi_m(r)\\&\quad+\mu(r)\left(\dfrac{1-a^2}{a(1-a\gamma)}\right)^2\sum_{n=m+1}^{\infty}\left(\dfrac{a(1-\gamma)}{a(1-a\gamma)}\right)^{2n}\left(\dfrac{\varphi_{2n}(r)}{1+\dfrac{a-\gamma}{1-a\gamma}}+\Phi_{2n+1}(r)\right).
 			\end{align*}
 			Next, it is a simple task to verify that
 			\[\lim\limits_{a\to 1^-}\left(\sum_{n=m+1}^{\infty}\dfrac{(1+a)}{a(1-a\gamma)}\left(\dfrac{a(1-\gamma)}{1-a\gamma}\right)^n\varphi_n(r)-\dfrac{2}{1-\gamma}\sum_{n=m+1}^{\infty}\varphi_n(r)\right)=0,\]
 			\[\lim\limits_{a\to 1^-}\left(p\left(\dfrac{1+\gamma}{1-\gamma}\right)+\frac{1}{1-a}\bigg(\left(\dfrac{a-\gamma}{1-a\gamma}\right)^p-1\bigg)\right)=0,\]
 			\[\lim\limits_{a\to 1^-}\left(\dfrac{1-a^2}{a(1-a\gamma)}\right)^2\sum_{n=m+1}^{\infty}\left(\dfrac{a(1-\gamma)}{a(1-a\gamma)}\right)^{2n}\left(\dfrac{\varphi_{2n}(r)}{1+\dfrac{a-\gamma}{1-a\gamma}}+\Phi_{2n+1}(r)\right)=0.\]
 			Thus, it follows that 
 			\begin{align*}
 				\mathcal{M}^\mu_{h_a}\left(\varphi,m,p,\gamma\right)=\varphi_0(r)+\dfrac{(1-a)}{1+\gamma}\left(2\sum_{n=m+1}^{\infty}\varphi_n(r)-p(1+\gamma)\varphi_m(r)\right)+O\left((1-a)^2\right)
 			\end{align*}
 			as $ a $ tends to $ 1^- $.\vspace{1.2mm} 
 			
 			In view of \eqref{Eq-BS-3.6}, it is easy to see that 
 			\[\mathcal{M}^\mu_{h_a}\left(\varphi,m,p,\gamma\right)>\varphi_m(r)\]
 			when $a$ is very close to $1$ \textit{i.e.} $a\rightarrow 1^{-}$ and $r\in\left(R_{\Omega_{\gamma}}(p,m),R_{\Omega_{\gamma}}(p,m)+\epsilon \right)$ which shows that the radius $R_{\Omega_{\gamma}}(p,m)$ cannot be improved further. This completes the proof.
 		\end{proof}
 		\begin{proof}[\bf Proof of Lemma \ref{lem-3.6}]
 			Without loss of generality, we assume that $ \gamma\in [0,1) $. Then it is easy to see that $ z\in\mathbb{D}_{\gamma}:=\mathbb{D}(\gamma;1-\gamma) $ if, and only if, $ w=(z-\gamma)/(1-\gamma)\in\mathbb{D}. $
 			Then we have 
 			\begin{align*}
 				g(z)=\sum_{n=0}^{\infty}A_n(1-\gamma)^n\phi^n(z)=\sum_{n=0}^{\infty}B_n\phi^n(z):=G(\phi(z))
 			\end{align*}
 			for $ z\in \mathbb{D}_{\gamma}, $ where $ B_n= A_{n} (1-\gamma)^n$ for $n\in \mathbb{N}_0$ and $A_0=\alpha_0I$ with $|\alpha_0|<1$. A simple computation shows that
 			\begin{align}\label{e-2.8}
 				S_{\rho}=\text{Area}\bigg(G(\mathbb{D}(0,\rho))\bigg)\leq (1-||B_0||^2)^2\frac{\rho^2}{(1-\rho^2)^2}=(1-|\alpha_0|^2)^2\frac{\rho^2}{(1-\rho^2)^2}.
 			\end{align}
 			Therefore,
 			\begin{align}\label{e-2.9}
 				\sum_{n=1}^{\infty}||A_n||\rho^n\leq\frac{1-|\alpha_0|^2}{1+\gamma}\sum_{n=1}^{\infty}\left(\frac{\rho}{1-\gamma}\right)^n=\frac{1-|\alpha_0|^2}{1+\gamma}\left(\frac{\rho}{1-\gamma-\rho}\right).
 			\end{align}
 			In view of \eqref{e-2.8} and \eqref{e-2.9}, we obtain 
 			\begin{align*} 
 				\sum_{n=0}^{\infty}||A_n||\rho^n+Q\left(S_{\rho}\right)&= |\alpha_0|+\frac{(1-|\alpha_0|^2)\rho}{(1+\gamma)(1-\gamma-\rho)}+\sum_{j=1}^{m}c_j\bigg(\frac{(1-|\alpha_0|^2)\rho}{(1-\rho^2)}\bigg)^{2j}\\&= 1+\Psi_1^{\gamma}(\rho),
 			\end{align*}
 			where 
 			\begin{align*}
 				\Psi_1^{\gamma}(\rho):=\frac{(1-|\alpha_0|^2)\rho}{(1+\gamma)(1-\gamma-\rho)}+\sum_{j=1}^{m}c_j\bigg(\frac{(1-|\alpha_0|^2)\rho}{(1-\rho^2)}\bigg)^{2j}-(1-|\alpha_0|).
 			\end{align*}
 			Let $ G_m(\rho) $ be defined by
 			\begin{align*}
 				G_m(\rho)&:=\left(\frac{c_{m-1}}{c_m}\right)\frac{\rho^{2m-2}}{(1-\rho^2)^{2m-2}(1-|\alpha_0|^2)^{2}}+\cdots+\left(\frac{c_{m-j}}{c_m}\right)\frac{\rho^{2m-2j}}{(1-\rho^2)^{2m-2j}(1-|\alpha_0|^2)^{2j}} \\[2mm]
 				&\quad\quad+\cdots+\left(\frac{c_1}{c_m}\right)\frac{\rho^2}{(1-\rho^2)^2(1-|\alpha_0|^2)^{2m-2}}+\frac{\rho^{2m}}{(1-\rho^2)^{2m}}.
 			\end{align*}
 			Then we can write $ \Psi_1^{\gamma}(\rho) $ as
 			\begin{align*} 
 				\Psi_1^{\gamma}(\rho)&=\frac{1-|\alpha_0|^2}{2}\bigg(1+2c_m(1-|\alpha_0|^2)^{2m-2}G_m(\rho)-\frac{2}{1+|\alpha_0|} \\&\quad+\left(\frac{2\rho}{(1+\gamma)(1-\gamma-\rho)}-1\right)\bigg).
 			\end{align*}
 			We suppose that $ \rho\leq \rho_0 $. Then it is easy to see that $\Psi_1^{\gamma}(\rho)$ is an increasing function and hence $ 	\Psi_1^{\gamma}(\rho)\leq 	\Psi_1^{\gamma}(\rho_0) $, where $ {2\rho_0}/{(1+\gamma)(1-\gamma-\rho_0)}=1 $ which is equivalent to $ \rho_0={(1-\gamma^2)}/{(3+\gamma)}. $\vspace{1.2mm} 
 			
 			A simple computation shows that 
 			\begin{align*}
 				\Psi_1^{\gamma}(\rho_0)=\frac{1-|\alpha_0|^2}{2}\left(1+2F_m(|\alpha_0|)-\frac{2}{1+|\alpha_0|}\right):=\frac{1-|\alpha_0|^2}{2}J(|\alpha_0|),
 			\end{align*}
 			where 
 			\begin{align*}
 				F_m(|\alpha_0|)&:=c_m(1-|\alpha_0|^2)^{2m-1}A^{2m}(\gamma)+c_{m-1}(1-|\alpha_0|^2)^{2m-3}A^{2(m-1)}(\gamma)+\cdots\\&\quad\quad+c_{m-j}(1-|\alpha_0|^2)^{2m-2j-1}A^{2(m-j)}(\gamma)+\cdots+c_1(1-|\alpha_0|^2)A^{2}(\gamma),\\
 				J(x)&:=1+2F_m(x)-\frac{2}{1+x}\;\; \mbox{for}\;\ x\in [0,1]\;\;\text{and}\;\;\\ A(\gamma)&:=\frac{(3+\gamma)(1-\gamma^2)}{(3+\gamma)^2-(1-\gamma^2)^2}.
 			\end{align*}
 			We note that $ A(\gamma)>0 $ for $ \gamma\in [0,1) $. In order to show that $ \Psi_1^{\gamma}(\rho_0)\leq 0 $, it is enough to show that $ J(x)\leq 0 $ for $ x\in [0,1] $. Further, a simple computation shows that
 			\begin{align*}
 				J(0)&=2c_mA^{2m}(\gamma)+2c_{m-1}A^{2(m-1)}(\gamma)+\cdots+c_{m-j}A^{2(m-j)}(\gamma)+\cdots+c_1A^2(\gamma)-1\\ \;\; \text{and}\;\; &\lim_{x\rightarrow 1^{-}}J(x)=0.
 			\end{align*}
 			It is easy to see that $ A(\gamma)=(f_1\circ f_2)(\gamma) $, where $ f_1(\rho)=\rho/(1-\rho^2) $ and $ f_2(\gamma)=(1-\gamma^2)/(3+\gamma). $
 			Since $ A^{\prime}(\gamma)=f^{\prime}_1(f_2(\gamma))f^{\prime}_2(\gamma) $ and 
 			\begin{equation}
 				f^{\prime}_2(\gamma)=-\left(\frac{\gamma^2+6\gamma+1}{(3+\gamma)^2}\right)<0,
 			\end{equation}
 			we show that $ f_1(\rho) $ is an increasing function of $ \rho $ in $ (0,1) $, and $ f_2 $ is a decreasing function of $ \gamma $ in $ [0,1) $ as well as $ A(\gamma) $ is a decreasing function of $ \gamma $  in $ [0,1) $, with $ A(0)=3/8 $ and $ A(1)=0 $. It can be seen that each $ A^{2j}(\gamma) $ is a decreasing function on $ [0,1) $, for $ j=1, 2, \cdots, m $. Thus it follows that 
 			\begin{align*}
 				A^{2j}(\gamma)\leq A^{2j}(0)=\left(\frac{3}{8}\right)^{2j}\;\; \mbox{for}\;\; j=1, 2, \cdots, m.
 			\end{align*}
 			Since $ x\in [0,1] $ and 
 			\begin{align*}
 				d_s=\max_{a\in [0,1]}\left(a(1+a)^2(1-a^2)^{2s-2}\right), s=2,\dots m,
 			\end{align*}
 			 a simple computation shows that
 			\begin{align*}
 				x(1+x)^2A^2(\gamma)&\leq 4\left(\frac{3}{8}\right)^2\\ x(1+x)^2(1-x^2)^2A^4(\gamma)\vspace{0.2in}&\leq d_2\left(\frac{3}{8}\right)^4\\&\vdots\\ x(1+x)^2(1-x^2)^{2m-2}A^{2m}(\gamma)&\leq d_m\left(\frac{3}{8}\right)^{2m}.
 			\end{align*}
 			It is easy to see that
 			\begin{align*}
 				J^{\prime}(x)&=\frac{2}{(1+x)^2}\bigg(1-2c_1x(1+x)^2A^2(\gamma)-6c_2 x(1+x)^2(1-x^2)^2A^4(\gamma)\\&\quad\quad-\cdots-2(2m-1)c_mx(1+x)^2(1-x^2)^{2m-2}A^{2m}(\gamma)\bigg)\\&\geq  \frac{2}{(1+x)^2}\left(1-\left(8c_1\left(\frac{3}{8}\right)^2+6c_2d_2\left(\frac{3}{8}\right)^4+\cdots+2(2m-1)c_md_m\left(\frac{3}{8}\right)^{2m}\right)\right)\\&\geq 0,\;\;\;\;\; \text{if}\;\; 8c_1\left(\frac{3}{8}\right)^2+6c_2d_2\left(\frac{3}{8}\right)^4+\cdots+2(2m-1)c_md_m\left(\frac{3}{8}\right)^{2m}\leq 1.
 			\end{align*}
 			Clearly, $ J(x) $ is an increasing function in $ [0,1] $ for 
 			\begin{equation*}
 				8c_1\left(\frac{3}{8}\right)^2+6c_2d_2\left(\frac{3}{8}\right)^4+\cdots+2(2m-1)c_md_m\left(\frac{3}{8}\right)^{2m}=1
 			\end{equation*}
 			and	hence, 
 			\begin{align*}
 				J(x)\leq \lim_{x\rightarrow 1^{-}}J(x)=0.
 			\end{align*}
 			Thus it follows that  $ J(x)\leq 0 $ for all $ x\in [0,1] $ and $ \gamma\in [0,1) $.  This completes the proof.
 		\end{proof}

 		\begin{proof}[\bf Proof of Theorem \ref{th-3.6}]
 			Let $f \in \mathcal{B}(\Omega_{\gamma})$ and $ g(z)=f((z-\gamma)/(1-\gamma)) $. Then it is easy to see that $ g\in\mathcal{B}(\mathbb{D})$ and 
 			\begin{align*}
 				g(z)=\sum_{n=0}^{\infty}\frac{1}{(1-\gamma)^n}A_n(z-\gamma)^n.
 			\end{align*}
 			Using Lemma \ref{lem-3.6}, we obtain 
 			\begin{align*}
 				\sum_{n=0}^{\infty}\frac{||A_n||}{(1-\gamma)^n}\rho^n+Q\left(S_{\rho}\right)\leq 1\;\; \text{for}\;\; \rho\leq\frac{1-\gamma^2}{3+\gamma}
 			\end{align*} which is equivalent to 
 			\begin{align}\label{e-4.15}
 				\sum_{n=0}^{\infty}{||A_n||}\left(\frac{\rho}{(1-\gamma)}\right)^n+Q\left(S_{\rho}\right)\leq 1\;\; \text{for}\;\; \rho\leq\frac{1-\gamma^2}{3+\gamma}.
 			\end{align}
 			Setting $ \rho=r(1-\gamma) $. In view of \eqref{e-4.15}, we obtain 
 			\begin{align*}
 				\sum_{n=0}^{\infty}{||A_n||}r^n+Q\left(S_{r(1-\gamma)}\right)\leq 1\;\; \text{for}\;\; r\leq r_0=\frac{1+\gamma}{3+\gamma}.
 			\end{align*} 
 			To show the sharpness of the result, we consider the following function $h_a$ given in \eqref{eeee-2.6}. A simple computation shows that
 			\begin{equation*}
 				h_a(z)=A_0-\sum_{n=1}^{\infty}A_nz^n\;\; \mbox{in}\;\; \mathbb{D},
 			\end{equation*} where $A_0$ and $A_n$ are given in \eqref{eee-2.7}. A simple computation shows that
 			\begin{align*} &
 				\sum_{n=0}^{\infty}{||A_n||}r^n+Q\left(S_{r(1-\gamma)}\right)\\&=\frac{a-\gamma}{1-a\gamma}+\left(\frac{1-a^2}{1-a\gamma}\right)\frac{(1-\gamma)r}{1-a\gamma-ar(1-\gamma)}+\frac{c_1r^2(1-a^2)^2(1-\gamma)^4}{((1-a\gamma)^2-a^2r^2(1-\gamma)^4)^2}\\[2mm]& \quad\quad+\frac{c_2r^4(1-a^2)^4(1-\gamma)^8}{((1-a\gamma)^2-a^2r^2(1-\gamma)^4)^4}+\cdots+\frac{c_mr^{2m}(1-a^2)^{2m}(1-\gamma)^{4m}}{((1-a\gamma)^2-a^2r^{2}(1-\gamma)^4)^{2m}} \\& \;=1-(1-a)\Phi^{\gamma}_1(r),
 			\end{align*}
 			where 
 			\begin{align*} 
 				\Phi^{\gamma}_1(r)&:=-\frac{(1+a)(1-\gamma)r}{(1-a\gamma-ar(1-\gamma))(1-a\gamma)}-\frac{c_1r^2(1-a)(1+a)^2(1-\gamma)^4}{((1-a\gamma)^2-a^2r^2(1-\gamma)^4)^2} \\[2mm]
 				&\quad\quad -\frac{c_2r^4(1-a)^3(1+a)^4(1-\gamma)^8}{((1-a\gamma)^2-a^2r^2(1-\gamma)^4)^4}-\frac{c_mr^{2m}(1-a)^{2m-1}(1+a)^{2m}(1-\gamma)^{4m}}{((1-a\gamma)^2-a^2r^{2}(1-\gamma)^4)^{2m}}\\[2mm]&\quad\quad-\frac{1}{1-a}\left(\frac{a-\gamma}{1+a\gamma}-1\right).
 			\end{align*}
 			By a straightforward calculation, it can be shown that $ \Phi^{\gamma}_1 $ is strictly decreasing function of $ r$ in $(0,1) $. Therefore,  for $ r>r_0=(1+\gamma)/(3+\gamma) $, we have $ \Phi^{\gamma}_1(r)<\Phi^{\gamma}_1(r_0) $. An elementary calculation shows that 
 			\begin{align*}
 				\lim_{a\rightarrow 1^-}\Phi^{\gamma}_1(r_0)=-\frac{2r_0}{(1-\gamma)(1-r_0)}+\frac{1+\gamma}{1-\gamma}=0.
 			\end{align*}
 			Hence,  $ \Phi^{\gamma}_1(r)< 0 $ for $ r>r_0 $. Consequently, we have $ 1-(1-a)\Phi^{\gamma}_1(r)>1 $ for $ r>r_0, $ which shows that $ r_0 $ is best possible. 
 		\end{proof}
 \section{\bf Bohr inequality for Bloch spaces on simply connected domain}
For $\nu\in (0,1),$ an analytic function $f$ in the unit disk $\mathbb{D}$ is called a $\nu$-Bloch function  if 
\begin{align*}
	\beta(\nu):=\sup_{z\in \mathbb{D}}(1-|z|^2)^\nu|f^{\prime}(z)|<\infty.
\end{align*}
The class of all $\nu$-Bloch functions is denoted by $\mathcal{B}(\nu).$ The class 	$\mathcal{B}(\nu)$ is a Banach space with respect to the Bloch norm $||.||_{\nu},$ which is defined by $||f||_\nu=|f(0)|+\beta(\nu).$ The $\nu$-Bloch space is a generalization of classical Bloch space $\mathcal{B}(1).$	\vspace{1.2mm}

The definition of $\nu$-Bloch space can be generalized to an arbitrary proper simply conceded domain $\Omega$ in $\mathbb{C}$ by means of hyperbolic metric. The hyperbolic metric $\lambda_{\mathbb{D}|z|}$ of $\mathbb{D}$ is defined by $\lambda_{\mathbb{D}}(z)=1/(1-|z|^2),$ $z\in \mathbb{D}.$ The equivalent form of the norm $\beta(\nu)$ for $\nu\in (0,\infty)$ in terms of hyperbolic density becomes 
\begin{align*}
	\beta(\nu)=\sup_{z\in \mathbb{D}}\dfrac{|f^{\prime}(z)|}{\lambda^{\nu}_{\mathbb{D}}(z)}.
\end{align*}
Let $\Omega\subset\mathbb{C}$ be a proper simply connected domain and $f:\Omega\rightarrow \mathbb{D}$ be a conformal mapping. Then the hyperbolic metric $\lambda_{\Omega}(z)|dz|$ of $\Omega$ is defined by \cite{Beardon-Minda} 
\begin{align}\label{EQ-BS-4.1}
	\lambda_{\Omega}(z):=\lambda_{\mathbb{D}}\left(f(z)\right)|f^{\prime}(z)|,\;\;z\in \Omega.
\end{align} 
We note that the definition \eqref{EQ-BS-4.1} is independent of the choice of conformal mapping of $\Omega$ onto $\mathbb{D}.$ Let $f$ be an analytic function in $\Omega\subset\mathbb{C}.$ Then $f$ is said to be an $\nu$-Bloch function in $\Omega$ if
\begin{align*}
	\beta_{\Omega}(\nu):=\sup_{z\in \Omega}\dfrac{|f^{\prime}(z)|}{\lambda^{\nu}_{\Omega}(z)}<\infty
\end{align*} 
and the space of all $\nu$-Bloch functions is denoted by $\mathcal{B}_{\Omega}(\nu).$ We define the norm on $\mathcal{B}_{\Omega}(\nu)$ by $||f||_{\Omega,\nu}=|f(0)|+\beta_{\Omega}(\nu).$ Anderson \emph{et al.} \cite{Anderson-Clunie-Pommerenke-JRAM-2006} established several results on the coefficients and zeros of Bloch functions and the boundary behavior of normal functions. In 1994, Bonk \emph{et al.} \cite{Bonk-Minda-Yanagihara-CMFT-1994} studied extensively the hyperbolic metric on Bloch regions. There has been a significant work on the sharp distortion estimates for locally univalent Bloch functions in \cite{Bonk-Minda-Yanagihara-CMFT-1996,Yanagihara-BLMS-1994} and references therein. Gnuschke-Hauschild and Pommerenke \cite{Gnuschke-Hauschild-Pommerenke-JRM} have established several interesting results on Bloch functions for gap series. Liu and Ponnusamy \cite{Liu-Ponnusamy-RM-2018} investigated subordination principles for harmonic Bloch mappings and Bohr's theorem for Bloch spaces. Allu and Halder \cite{Allu-Halder-arxiv-2022} studied Bloch-Bohr radius for $\alpha$-Bloch mappings on a proper simply connected domain $\Omega\subset\mathbb{C}$ for $\alpha\in (0,\infty).$\vspace{1.2mm}

In this consequence, we define  $\nu$-Bloch functions of operator valued holomorphic functions on simply connected $\Omega$ in $\mathbb{C}.$
For $\nu\in (0,1),$ a function $ f\in {H}^{\infty}\left(\Omega, \mathcal{B}(\mathcal{H})\right)$ be given by $ f(z)=\sum_{s=0}^{\infty}A_sz^s $ in $ \mathbb{D} $ is called a $\nu$-Bloch function in $\Omega$ if 
\begin{align*}
	\beta^*_{\Omega}(\nu):=\sup_{z\in \Omega}\dfrac{||Df(z)||}{\lambda^{\nu}_{\Omega}(z)}<\infty,
\end{align*} 
where  $\lambda_{\Omega}(z)|dz|$ of $\Omega$ is the hyperbolic metric, defined by  
\begin{align*}
	\lambda_{\Omega}(z):=\lambda_{\mathbb{D}}\left(f(z)\right)||Df(z)||,\;\;z\in \Omega
\end{align*}
and $Df(z)$ is denotes the derivative of $f$ at $z.$ Let $\mathcal{B}^*_{\Omega}(\nu)$ be the class of all $\nu$-Bloch functions $ f\in {H}^{\infty}\left(\Omega, \mathcal{B}(\mathcal{H})\right)$ be given by $ f(z)=\sum_{s=0}^{\infty}A_sz^s $ in $ \mathbb{D}, $ where $A_s\in \mathcal{B}(\mathcal{H})$ for $s\in \mathbb{N}_0.$ \vspace{1.2mm}

However, it is important to note that while Bohr-type inequalities for many classes of functions have been extensively explored, the study of these inequalities for the class $\mathcal{B}^*_{\Omega}(\nu)$, which consists of all $\nu$-Bloch functions $f$ in $H^{\infty}(\Omega, \mathcal{B}(\mathcal{H}))$ of the form $f(z)=\sum_{s=0}^{\infty}A_sz^s$ in $\mathbb{D}$, where $A_s\in \mathcal{B}(\mathcal{H})$ for $s\in \mathbb{N}_0$, has not received as much attention from researchers. This lack of attention serves as the primary motivation to study the Bloch-Bohr radius for $\nu$-Bloch mappings on a proper, simply connected domain $\Omega\subset\mathbb{C}$. Our main goal is to fill this specific gap in the existing literature and contribute to the understanding of Bohr-type inequalities.\vspace{1.2mm}

In this section, we obtain the following results which are Bohr inequalities when $f\in \mathcal{B}^*_{\Omega}(\nu)$ or $f\in \mathcal{B}^*_{\Omega_{\gamma}}(\nu)$.
\begin{thm}\label{Th-BS-4.1}
	Let $\Omega$ be a proper simply connected domain containing $\mathbb{D}.$ Let $f\in \mathcal{B}^*_{\Omega}(\nu)$  with $||f||_{\Omega,\nu}\leq 1$ such that $f(z)=\sum_{s=0}^{\infty}A_sz^s$ in $\mathbb{D},$ where $A_s\in \mathcal{B}(\mathcal{H})$ for $s\in \mathbb{N}_0.$ Then we obtain
	\begin{align*}
		\sum_{s=0}^{\infty}||A_s||r^s\leq 1\;\;\mbox{for}\;\;|z|=r\leq R_{\Omega}(\nu),
	\end{align*}
	where $R_{\Omega}(\nu)$ is the smallest root in $(0,1)$ of the equation 
	\begin{align}\label{Eq-BS-4.1}
		M(r):=\dfrac{r}{2\pi}\int_{|z|=r}\lambda^{2\nu}_{\Omega}(z)|dz|=\dfrac{6}{\pi^2},
	\end{align}
	provided $\lim_{r\rightarrow 1^-}M(r)>6/\pi^2.$
\end{thm}
\begin{thm}\label{Th-BS-4.2}
	For $0\leq \gamma<1,$ let $f\in \mathcal{B}^*_{\Omega_{\gamma}}(\nu)$ with $||f||_{\Omega_{\gamma},\nu}$ such that $f(z)=\sum_{s=0}^{\infty}A_sz^s$ in $\mathbb{D},$ where $A_s\in \mathcal{B}(\mathcal{H}).$ Than we obtain 
	\begin{align*}
		\sum_{s=0}^{\infty}||A_s||r^s\leq 1 \;\;\mbox{for}\;\; |z|=r\leq R_{\Omega_{\gamma}}(\nu),
	\end{align*}
	where $R_{\Omega_{\gamma}}(\nu)$ is the unique root in $(0,1)$ of the equation $N_{\Omega_{\gamma},\nu}(r)=0$ and 
	\begin{align}\label{Eq-BS-4.6}
		N_{\Omega_{\gamma},\nu}(r):={(1-\gamma)^{2\nu}r^2\pi^2}-{6{(1-((1-\gamma)r+\gamma)^2)^{2\nu}}}.
	\end{align}
\end{thm}
Our next result is an improved version of the Bohr inequality concerning $S_r(f)$, the planer integral of $\mathbb{D}_r$ under the mapping $f$.
\begin{thm}\label{Th-BS-4.3}
	Let $\Omega$ be a proper simply connected domain containing $\mathbb{D}.$ Let $f\in \mathcal{B}^*_{\Omega}(\nu)$  with $||f||_{\Omega,\nu}\leq 1$ such that $f(z)=\sum_{s=0}^{\infty}A_sz^s$ in $\mathbb{D},$ where $A_s\in \mathcal{B}(\mathcal{H})$ for $s\in \mathbb{N}_0.$ Let $\mathbb{D}_r:=\{z\in \mathbb{C}:|z|<r\}\subset\Omega$ and $S_r(f)$ be the planer integral of $\mathbb{D}_r$ under the mapping $f$ and $\mu:[0,1]\rightarrow[0,\infty)$ be a continuous function. Then we obtain
	\begin{align}\label{Eq-BS-4.10}
		||f(z)||+\sum_{s=1}^{\infty}||A_s||r^s+\mu(r)\left(S_r(f)\right)\leq 1\;\;\mbox{for}\;\;|z|=r\leq R_{\Omega,f}(\nu),
	\end{align}
	where $R_{\Omega,f}(\nu)$ is the smallest root in $(0,1)$ of the equation $H_{\Omega,\nu}(r)=0,$  	provided $\lim_{r\rightarrow 1^-}H_{\Omega,\nu}(r)>0,$ where
	\begin{align}\label{Eq-BS-4.11}
		H_{\Omega,\nu}(r):=r\int_{|z|=r}\lambda^{2\nu}_{\Omega}(z)|dz|-\dfrac{3}{\pi}.
	\end{align}
\end{thm}
\begin{proof}[\bf Proof of Theorem \ref{Th-BS-4.1}]
	Let $f\in \mathcal{B}^*_{\Omega}(\nu)$  with $||f||_{\Omega,\nu}\leq 1$	and $a=||A_0||=||f(0)||.$
	Then \begin{align*}
		||f(0)||+\sup_{z\in \Omega}\dfrac{||Df(z)||}{\lambda^{\nu}_{\Omega}(z)}\leq 1
	\end{align*}	
	which implies that $	{||Df(z)||}\leq (1-a){\lambda^{\nu}_{\Omega}}(z).$ Therefore 
	\begin{align}\label{Eq-BS-4.2}
		{||Df(z)||}^2\leq (1-a)^2{\lambda^{2\nu}_{\Omega}}(z)\;\;\mbox{for}\;\;z\in\Omega.
	\end{align}	
	Since $f(z)=\sum_{s=0}^{\infty}A_sz^s$ in $\mathbb{D},$ using $\eqref{Eq-BS-4.2}$ we obtain 
	\begin{align}\label{Eq-BS-4.3}
		\bigg|\bigg|\sum_{s=1}^{\infty}sA_sz^{s-1}\bigg|\bigg|^2   ={||Df(z)||}^2\leq (1-a)^2{\lambda^{2\nu}_{\Omega}}(z).
	\end{align}
	Integrating \eqref{Eq-BS-4.3} over the circle $|z|=r<1,$ we obtain 
	\begin{align*}
		2\pi r\sum_{s=1}^{\infty}s^2||A_s||^2r^{2(s-1)} \leq (1-a)^2\int_{|z|=r}\lambda^{2\nu}_{\Omega}(z)|dz|\;\;\mbox{for}\;\; z\in \mathbb{D}
	\end{align*}
	which leads to
	\begin{align}\label{Eq-BS-4.4}
		\sum_{s=1}^{\infty}s^2||A_s||^2r^{2s} \leq (1-a)^2\dfrac{r}{2\pi}\int_{|z|=r}\lambda^{2\nu}_{\Omega}(z)|dz|\;\;\mbox{for}\;\; z\in \mathbb{D}.
	\end{align}
	In view of the classical Cauchy-Schwartz inequality and \eqref{Eq-BS-4.4}, we obtain 
	\begin{align*}
		||A_0||+\sum_{s=1}^{\infty}||A_s||r^s&\leq a+\sqrt{\sum_{s=1}^{\infty}s^2||A_s||^2r^{2s}}\sqrt{\sum_{s=1}^{\infty}\dfrac{1}{s^2}}\\&\leq a+(1-a)\sqrt{\dfrac{r}{2\pi}\int_{|z|=r}\lambda^{2\nu}_{\Omega}(z)|dz|}\sqrt{\dfrac{\pi^2}{6}}\\&= 1+(1-a)\left(-1+\sqrt{\dfrac{r}{2\pi}\int_{|z|=r}\lambda^{2\nu}_{\Omega}(z)|dz|}\sqrt{\dfrac{\pi^2}{6}}\right).
	\end{align*}
	The right hand side of the above inequality less than or equal to $1$ if 
	\begin{align*}
		\dfrac{r}{2\pi}\int_{|z|=r}\lambda^{2\nu}_{\Omega}(z)|dz|\leq \dfrac{6}{\pi^2}
	\end{align*}
	which holds for $r\leq R_{\Omega}(\nu),$ where $R_{\Omega}(\nu)$ is the smallest root in $(0,1)$ of the  $\eqref{Eq-BS-4.1}.$\vspace{1.2mm}
	
	In order to prove the existence of the root $ R_{\Omega}(\nu),$ we now consider the function $G: [0,1)\rightarrow \mathbb{R}$ defined by $G(r):=M(r)-6/\pi^2.$  It is easy to see that $G$ is a continuous function in $[0,1)$ with 
	\begin{align}\label{Eq-BS-4.5}
		G(0)=M(0)-\dfrac{6}{\pi^2}=-\dfrac{6}{\pi^2}<0\;\;\mbox{and}\;\; \lim_{r\rightarrow 1^-}G(r)=\lim_{r\rightarrow 1^-}M(r)-\dfrac{6}{\pi^2}>0.
	\end{align}
	Thus, in view of $\eqref{Eq-BS-4.5}$, we conclude that $G$ has a root in $(0,1)$ and we choose the smallest root to be $R_{\Omega}(\nu).$ This completes the proof.
\end{proof}	
		
\begin{proof}[\bf Proof of Theorem \ref{Th-BS-4.2}]
	Let $a=||f(0)||.$ The hyperbolic density for $\Omega_{\gamma}$ at $z$ is 
	\begin{align*}
		\lambda_{\Omega_{\gamma}}(z)=\dfrac{1-\gamma}{1-|(1-\gamma)z+\gamma|^2},\;\;z\in \Omega_{\gamma}.
	\end{align*}
	Then by the hypothesis $||f||_{\Omega_{\gamma},\nu}\leq 1,$ we obtain 
	\begin{align*}
		||Df(z)||\leq \dfrac{(1-a)(1-\gamma)^{\nu}}{\left(1-|(1-\gamma)z+\gamma|^2\right)^{\nu}}\leq \dfrac{(1-a)(1-\gamma)^{\nu}}{\left(1-((1-\gamma)|z|+\gamma)^2\right)^{\nu}},\;\;z\in \Omega_{\gamma}.
	\end{align*}
	Thus it follows that
	\begin{align}\label{Eq-BS-4.66}
		\bigg|\bigg|\sum_{s=1}^{\infty}sA_sz^{s-1}\bigg|\bigg|^2   ={||Df(z)||}^2\leq \dfrac{(1-a)^2(1-\gamma)^{2\nu}}{\left(1-((1-\gamma)|z|+\gamma)^2\right)^{2\nu}}.
	\end{align}
	Integrating \eqref{Eq-BS-4.6} over the circle $|z|=r<1$ gives that
	\begin{align*}
		2r\pi\sum_{s=1}^{\infty}s^2||A_s||^2r^{2(s-1)}&\leq (1-a)^2\int_{|z|=r}\dfrac{(1-\gamma)^{2\nu}}{\left(1-((1-\gamma)|z|+\gamma)^2\right)^{2\nu}}|dz|\\&=(1-a)^2\int_{\theta=0}^{2\pi}\dfrac{(1-\gamma)^{2\nu}}{\left(1-((1-\gamma)|re^{i\theta}|+\gamma)^2\right)^{2\nu}}rd{\theta}\\&=\dfrac{(1-a)^22\pi r(1-\gamma)^{2\nu}}{\left(1-((1-\gamma)r+\gamma)^2\right)^{2\nu}}
	\end{align*}
	which is equivalent to 
	\begin{align}\label{Eq-BS-4.7}
		\sum_{s=1}^{\infty}s^2||A_s||^2r^{2s}\leq \dfrac{(1-a)^2r^2(1-\gamma)^{2\nu}}{\left(1-((1-\gamma)r+\gamma)^2\right)^{2\nu}}.
	\end{align}
	In view of the classical Cauchy-Schwartz inequality and \eqref{Eq-BS-4.7}, we see that
	\begin{align*}
		||A_0||+\sum_{s=1}^{\infty}||A_s||r^s&\leq a+\sqrt{\sum_{s=1}^{\infty}s^2||A_s||^2r^{2s}}\sqrt{\sum_{s=1}^{\infty}\dfrac{1}{s^2}}\\&\leq a+\dfrac{(1-a)r(1-\gamma)^{\nu}}{\left(1-((1-\gamma)r+\gamma)^2\right)^{\nu}}\sqrt{\dfrac{\pi^2}{6}}\\& \leq1
	\end{align*}
	provided 
	\begin{align}\label{Eq-BS-4.8}
		\dfrac{r(1-\gamma)^{\nu}}{\left(1-((1-\gamma)r+\gamma)^2\right)^{\nu}}\sqrt{\dfrac{\pi^2}{6}}\leq 1.
	\end{align}
	The inequality $\eqref{Eq-BS-4.8}$ holds for $r\leq R_{\Omega_{\gamma}}(\nu),$ where $R_{\Omega_{\gamma}}(\nu)$ is the smallest root in $(0,1)$ of the equation $N_{\Omega_{\gamma},\nu}(r)=0$ given by \eqref{Eq-BS-4.66}.\vspace{1.2mm}
	
	To prove the existence of the root $R_{\Omega_{\gamma}(\nu)}$, we see that $N_{\Omega_{\gamma},\nu}(r)$ is a continuous function of $r$ in $[0,1]$ and satisfies the conditions
	\begin{align*}
		N_{\Omega_{\gamma},\nu}(0)=-6(1-\gamma^2)^{2\nu}<0\;\;\mbox{and}\;\; N_{\Omega_{\gamma},\nu}(1)=(1-\gamma)^{2\nu}\pi^2>0.
	\end{align*}
	Thus, applying the intermediate value theorem to the continuous function $N_{\Omega_{\gamma},\nu}$ in $[0,1]$, we conclude that $N_{\Omega_{\gamma},\nu}(r)=0$ has a root in $(0,1)$ and we choose the smallest root to be $R_{\Omega_{\gamma}}(\nu).$ This completes the proof. 
\end{proof}
\begin{proof}[\bf Proof of Theorem \ref{Th-BS-4.3}]
	Let $f\in \mathcal{B}^*_{\Omega}(\nu)$  with $||f||_{\Omega,\nu}\leq 1$	and $a=||f(0)||.$
	Using \eqref{Eq-BS-4.2}, we obtain 
	\begin{align}\label{Eq-BS-4.12}
		S_r(f)&=\int_{\mathbb{D}_r}||Df(z)||^2dA(z)\leq (1-a)^2\int_{\mathbb{D}_r}{\lambda^{2\nu}_{\Omega}}(z)dA(z).
	\end{align}
	In view of the classical Cauchy-Schwartz inequality and \eqref{Eq-BS-4.4}, \eqref{Eq-BS-4.12}, we obtain 
	\begin{align*}
		&||f(z)||+\sum_{s=1}^{\infty}||A_s||r^s+\mu(r)\left(S_r(f)\right)\\&=\bigg|\bigg|\sum_{s=0}^{\infty}A_sz^s\bigg|\bigg|+\sum_{s=1}^{\infty}||A_s||r^s+\mu(r)\left(S_r(f)\right)\\&\leq a+2\sum_{s=1}^{\infty}||A_s||r^s+\mu(r)(1-a)^2\int_{\mathbb{D}_r}{\lambda^{2\nu}_{\Omega}}(z)dA(z)\\&\leq a+2\sqrt{\sum_{s=1}^{\infty}s^2||A_s||^2r^{2s}}\sqrt{\sum_{s=1}^{\infty}\dfrac{1}{s^2}}+\mu(r)(1-a)^2\int_{\mathbb{D}_r}{\lambda^{2\nu}_{\Omega}}(z)dA(z)\\&\leq a+2(1-a)\sqrt{\dfrac{r}{2\pi}\int_{|z|=r}\lambda^{2\nu}_{\Omega}(z)|dz|}\sqrt{\dfrac{\pi^2}{6}}+\mu(r)(1-a)^2\int_{\mathbb{D}_r}{\lambda^{2\nu}_{\Omega}}(z)dA(z)\\&= 1+(1-a)\mathcal{G}_{\Omega,\mu,\nu}(a,r),
	\end{align*}
	where
	\begin{align*}
		\mathcal{Q}_{\Omega,\mu,\nu}(a,r):=-1+2\sqrt{\dfrac{r}{2\pi}\int_{|z|=r}\lambda^{2\nu}_{\Omega}(z)|dz|}\sqrt{\dfrac{\pi^2}{6}}+\mu(r)(1-a)\int_{\mathbb{D}_r}{\lambda^{2\nu}_{\Omega}}(z)dA(z).
	\end{align*}
	Taking $a$ close to $1$, we obtain
	\begin{align*}
		\lim\limits_{a\rightarrow 1^{-}}\mathcal{Q}_{\Omega,\mu,\nu}(a,r):=-1+2\sqrt{\dfrac{r}{2\pi}\int_{|z|=r}\lambda^{2\nu}_{\Omega}(z)|dz|}\sqrt{\dfrac{\pi^2}{6}}.
	\end{align*}
	The right hand side of the above inequality less than or equal to $1$ if 
	\begin{align*}
		r\int_{|z|=r}\lambda^{2\nu}_{\Omega}(z)|dz|\leq \dfrac{3}{\pi}.
	\end{align*}
	which is holds for $r\leq R_{\Omega,f}(\nu),$ where $R_{\Omega,f}(\nu)$ is the smallest root in $(0,1)$ of the equation given by $\eqref{Eq-BS-4.11}.$\vspace{1.2mm}
	
	In order to prove the existence of the root $ R_{\Omega,f}(\nu),$ we consider the function $H_{\Omega,\nu}:[0,1)\rightarrow \mathbb{R}$ defined by \eqref{Eq-BS-4.11}.  It is easy to see that $H_{\Omega,\nu}$ is a continuous function in $[0,1)$ with 
	\begin{align}\label{Eq-BS-4.13}
		H_{\Omega,\nu}(0)=-\dfrac{3}{\pi}<0\;\;\mbox{and}\;\; \lim_{r\rightarrow 1^-}H_{\Omega,\nu}(r)>0.
	\end{align}
	Therefore, in view of $\eqref{Eq-BS-4.13}$, we conclude that $H_{\Omega,\nu}$ has a root in $(0,1)$ and choose the smallest root to be $R_{\Omega,f}(\nu).$ This completes the proof.
\end{proof}

\section{\bf Concluding remarks}
The study of the Bohr inequality for different classes of functions is currently an active research area in Geometric Functions Theory. In this paper, we have successfully established sharp versions of the Bohr inequality for operator-valued holomorphic functions on simply connected domains. The research work in this paper also extends to generalizing the refined Bohr inequality and the Bohr-Rogosinski inequality, involving a sequence of non-negative continuous functions that converge locally uniformly. We have established the sharpness of our results, which may be considered a certain foundation for further exploration in the Bohr radius problem.\vspace{2mm}

Furthermore, we have explored our findings to include the class of operator-valued $\nu$-Bloch functions in two distinct connected domains, $\Omega$ and $\Omega_{\gamma}$, in the complex plane, without knowledge of the coefficient bounds. Although we have not proven the Growth theorem for the class of operator-valued $\nu$-Bloch functions, which remains an open problem in geometric functions theory, our study of the Bohr-Rogosinski inequality and its contributions to the understanding of Bohr-type inequalities opens new avenues for research in complex analysis and operator theory.\vspace{4mm}

\noindent{\bf Acknowledgment:} The research of the first author is supported by UGC-JRF (Ref. No. 201610135853), New Delhi, Govt. of India and second author is supported by SERB File No. SUR/2022/002244, Govt. of India.

\section{\bf Declarations}
\noindent {\bf Funding:} Not Applicable.\vspace{1.5mm}

\noindent\textbf{Conflict of interest:} The authors declare that there is no conflict  of interest regarding the publication of this paper.\vspace{1.5mm}

\noindent\textbf{Data availability statement:}  Data sharing not applicable to this article as no datasets were generated or analysed during the current study.\vspace{1.5mm}

\noindent{\bf Code availability:} Not Applicable.\vspace{1.5mm}

\noindent {\bf Authors' contributions:} All the authors have equal contributions.

\end{document}